\documentclass[11pt]{amsart}

\usepackage{fullpage}
\usepackage[utf8]{inputenc}
\usepackage[T1]{fontenc}
\usepackage{graphicx,subfigure}
\usepackage{amsmath,amsfonts,amssymb,mathtools}
\usepackage{stmaryrd}
\usepackage{mathrsfs,dsfont}
\usepackage{amsthm}
\usepackage{mathabx}
\usepackage{tabularx}
\usepackage{float}

\usepackage{hyperref}
\usepackage{appendix}

\newcommand{\norme}[1]{\left\Vert #1\right\Vert}

\usepackage{enumerate}

\usepackage{color}

\newcommand{\E}{\mathbb{E}}
\newcommand{\IE}{\mathbb{E}}
\newcommand{\ts}{\thinspace}

\newcommand{\R}{\mathbb{R}}
\newcommand{\N}{\mathbb{N}}

\newtheorem{theo}{Theorem}[section]

\newtheorem{cor}[theo]{Corollary}

\newtheorem{propo}[theo]{Proposition}
\newtheorem{lemma}[theo]{Lemma}
\newtheorem{defi}[theo]{Definition}
\newtheorem{remark}[theo]{Remark}

\newtheorem{ass}{Assumption}

\usepackage{algorithm}
\usepackage{algorithmic}

\begin{document}

\title{
Preconditioning for the high-order sampling of the invariant distribution of parabolic semilinear SPDEs
}

\author{Charles-Edouard~Br\'ehier}
\address{Universite de Pau et des Pays de l’Adour, E2S UPPA, CNRS, LMAP, Pau, France
}
\email{Charles-Edouard.Brehier@univ-pau.fr 
}

\author{Adrien Busnot Laurent}
\address{Univ Rennes, INRIA (MINGuS), IRMAR - UMR 6625, ENS Rennes, France
}
\email{Adrien.Busnot-Laurent@inria.fr}

\author{Arnaud Debussche}
\address{Univ Rennes, CNRS, IRMAR - UMR 6625, ENS Rennes, France
}
\email{Arnaud.Debussche@ens-rennes.fr}

\author{Gilles Vilmart}
\address{Section of Mathematics, University of Geneva, Switzerland
}
\email{Gilles.Vilmart@unige.ch}

\date{\today}

\begin{abstract}
For a class of ergodic parabolic semilinear stochastic partial differential equations (SPDEs) with gradient structure,
we introduce a preconditioning technique and design high-order integrators for the approximation of the invariant distribution.
The preconditioning yields improved temporal regularity of the dynamics while preserving the invariant distribution and allows the application of postprocessed integrators.
For the semilinear heat equation driven by space-time white noise in dimension $1$, we obtain new temporal integrators with orders $1$ and $2$ for sampling the invariant distribution with a minor overcost compared to the standard semilinear implicit Euler method of order $1/2$.
Numerical experiments confirm the theoretical findings and illustrate the efficiency of the approach.

\medskip


\medskip

\end{abstract}

\maketitle

\section{Introduction}

The goal of the present paper is to estimate numerically integrals in infinite dimension of the form
\[
\int_H \varphi d\mu_\star=\underset{t\to\infty}\lim~\E\bigl[\varphi(X(t))\bigr],
\]
where the probability distribution $\mu_\star$ is defined on a separable Hilbert space $H$, $\varphi:H\to\R$ is a function with sufficient regularity, and $X=\bigl(X(t)\bigr)_{t\ge 0}$ is an ergodic $H$-valued Markov process with unique invariant distribution $\mu_\star$. Precisely, we consider stochastic evolution equations of the form
\begin{equation}\label{eq:semilin}
dX(t)=AX(t)dt+F(X(t))dt+dW(t),
\end{equation}
where $\bigl(W(t)\bigr)_{t\ge 0}$ is a cylindrical Wiener process, and the nonlinear term $F:H\to H$ is a Lipschitz continuous function on $H$. We assume a gradient structure: one has
\[
F(x)=-DV(x),\quad\forall~x\in H,
\]
for a given potential $V:H\to \R$. The linear operator $A$ is assumed to be self-adjoint, negative, and to be such that $Q=(-A)^{-1}$ is a trace-class operator.
The class of problems \eqref{eq:semilin} encompasses in particular parabolic semilinear stochastic partial differential equations (SPDEs),
\begin{equation} \label{eq:heat1d}
\frac{\partial x(t,z)}{\partial t}=\frac{\partial^2 x(t,z)}{\partial z^2}+f\bigl(x(t,z)\bigr)+\dot{W}(t,z),\quad t>0,~z\in(0,1),
\end{equation}
driven by space-time white noise, with homogeneous Dirichlet boundary conditions, and for a smooth and Lipschitz nonlinearity $f:\R\rightarrow \R$. In this case, one has $H=L^2(0,1)$, and for all $x\in H$, $V(x)=-\int_{0}^{1}U(x(z))dz$, where $U'=f$.
In the absence of nonlinearity $F=0$, the probability distribution $\mu_\star$ coincides with the Gaussian distribution $\nu=\mathcal{N}(0,\frac12 Q)$ with mean zero and covariance operator $\frac12 Q$. Under technical assumptions, the invariant distribution $\mu_\star$ is absolutely continuous with respect to the Gaussian distribution $\nu$ with a Gibbs density,
\begin{equation} \label{eq:def_mu_star}
d\mu_\star(x)=Z^{-1}\exp\bigl(-2V(x)\bigr)d\nu(x),
\end{equation}
where $Z=\int_H e^{-2V(x)}d\nu(x)$ is the normalization constant.
The aim of this paper is to design numerical integrators for sampling the invariant distribution $\mu_\star$ with high order of convergence with respect to the discretisation timestep in spite of the low regularity of the solution of the dynamics.

Consider an integrator given by a discrete-time Markov process
\begin{equation}
\label{equation:one_step_integrator}
Y_{n+1}^{\Delta t}=\psi^{\Delta t}(Y_n^{\Delta t}),
\end{equation}
where $\psi^{\Delta t}:H\to H$ is a random mapping depending on the time-step size $\Delta t$.
Assuming ergodicity of the Markov process $Y_n^{\Delta t}$ with the invariant distribution $\mu^{\Delta t}$, where we recall that $\mu^{\Delta t}\neq \mu_\star$ in general, we define the error for the invariant distribution
\[
\epsilon(\varphi,\Delta t)=\int_H \varphi d\mu^{\Delta t}-\int_H \varphi d\mu_\star,
\]
for sufficiently smooth test functions $\varphi:H\to\R$.
Note that a spatial discretization is also required in practice, however we mostly focus on the temporal discretization error in this article.
The discretization is said to be of order $q$ in time for the approximation of the invariant distribution, if for any sufficiently smooth function $\varphi:H\to\R$, there exists $C(\varphi)\in(0,\infty)$ such that for any time-step size $\Delta t$, one has
\[
|\epsilon(\varphi,\Delta t)|\le C(\varphi)\Delta t^q.
\]
Recall that the discretization is said to be of weak order $p$ if for any time $T\in(0,\infty)$ and any sufficiently smooth function $\varphi:H\to\R$, there exists $C(T,\varphi)\in(0,\infty)$ such that for any time-step size $\Delta t=T/N$ where $N\in\N$, one has
\[
\big|\E[\varphi(Y_N^{\Delta t})]-\E[\varphi(X(N\Delta t))]\big|\le C(T,\varphi)\Delta t^{p}.
\]
Under mild assumptions, the order $q$ for the invariant distribution is larger than the weak order $p$, i.e. $q\geq p$. Indeed, in the weak error estimates above, one is able to show that $C(T,\varphi)$ may be chosen independent of $T$.

The main contribution of this article is to propose and analyse integrators of order $q=1$ and $q=2$ for the approximation of the invariant distribution. This is a non-trivial result as for parabolic semilinear SPDEs, the weak order for the linear implicit Euler scheme is $p=\frac12^-$ (which is consistent with the $\frac14^-$ H\"older regularity of the trajectories). We refer for instance to~\cite{Andersson_Kruse_Larsson:16, Brehier:17, Brehier_Debussche:18,Brehier-Goudenege:20,Brehier_Hairer-Stuart:18,Conus_Jentzen_Kurniawan:19, Debussche:11, Debussche_Printems:09} for works on weak convergence of numerical schemes applied to semilinear parabolic SPDEs. Concerning the numerical approximation of the invariant distributions of such SPDEs, preliminary works are~\cite{Brehier:14, Brehier_Kopec:17}, where uniform in time weak error estimates have been obtained, for temporal and full discretization respectively. We also refer to more recent works~\cite{Brehier:22,Chen-Gan-Wang:20,Cui-Hong-Sun:21,Jiang-Wang,Liu:25,Wang-Cao} for similar results, in particular for equations with non-globally Lipschitz continuous nonlinearities. The article~\cite{Chen-Dang-Hong-Zhou:23} provides central limit theorems instead of weak error estimates. The recent work~\cite{Brehier:25focm} deals with the same type of Gibbs invariant distributions as in this manuscript, with the introduction and analysis of a modified Euler scheme which preserves regularity of trajectories.

Finally, the question of the numerical approximation of invariant distributions has been studied for other types of SPDEs: Schr\"odinger equations~\cite{Chen-Hong-Wang:17,Hong-Wang:19,Hong-Wang-Zhang:17}, Maxwell equations~\cite{Chen-Hong-Ji-Liang:25,Chen-Hong-Liang}, damped wave equations~\cite{Lei-Brehier-Gan:25,Cai-Chen-Hong-Zhou}, Burgers equations~\cite{Boyaval-Martel-Reygner:22}, Navier--Stokes equations~\cite{GlattHoltz-Mondaini:25,Wu-Wang-Huang:23}, Burgers--Huxley equations~\cite{Wang-Wu-Huand-Zheng:25,Wang-Zhang-Cao:25}, Cahn--Hilliard equations~\cite{Deng-Wang-Cao}.

To the best of our knowledge, no higher-order method with a rigorous analysis is known in the literature for semi-linear SPDEs of the form \eqref{eq:semilin}. We mention the first attempt~\cite{Brehier_Vilmart:16} on the construction of higher-order methods for the approximation of the invariant distribution, where new schemes were designed. Numerical evidence suggested that one may achieve order $q\ge 1$, and in some cases, order $q=\frac32^-$ has been proved in the simplified case of a linear function $F$.

The construction of the higher-order methods presented in this article is based on a preconditioning technique. Applying a suitable preconditioning operator $\mathcal{P}$ allows us to modify the dynamics~\eqref{eq:semilin} without altering the invariant distribution \eqref{eq:def_mu_star}. In general, one obtains the preconditioned dynamics,
\[
dY^{\mathcal{P}}(t)=\mathcal{P}\Bigl(AY^{\mathcal{P}}(t)+F(Y^{\mathcal{P}}(t))\Bigr)dt+\mathcal{P}^{\frac12}dW(t).
\]
In this paper, we shall focus on the choice $\mathcal{P}=(-A)^{-1}=Q$, which appears already in \cite{Cotter_Roberts_Stuart_White:13, ReyBellet_Spiliopoulos_2016, Hairer_Stuart_Vollmer:14} in the context of Markov Chain Monte-Carlo (MCMC) methods
or for the
sampling of conditioned diffusions using SPDE techniques \cite{Hairer_Stuart_Voss:09,Hairer-Stuart-Voss:07,Hairer-Stuart-Voss:05}. There has recently been a growing interest for preconditioning techniques for finite dimensional stochastic systems, in the context of high-dimensional Langevin dynamics in molecular dynamics
\cite{Best_Hummer_2010, Hummer2005PositiondependentDC} and in particular
for overdamped Langevin dynamics, see \cite{ReyBellet_Spiliopoulos_2016, Roberts_Stramer_2002} and recently \cite{Bronasco25els, Cui_Tong_Zahm_2024,Lelievre_Pavliotis_Robin_Santet_Stoltz_2024,Lelievre_Santet_Stoltz_2024,Lelievre_Santet_Stoltz_2024B}, where preconditioning permits to enhance the convergence to equilibrium in the transient phase.
Under mild assumptions, the stochastic evolution equation
\[
dY(t)=-Y(t)dt+QF(Y(t))dt+Q^{\frac12}dW(t)
\]
is also ergodic, with the same invariant distribution $\mu_\star$ as the original dynamics. Note that the preservation of the invariant dynamics when preconditioning the dynamics crucially relies on the gradient structure condition $F=-DV$. As will be clearly explained below, the preconditioning technique permits to increase the temporal regularity of the trajectories, while the spatial regularity is not modified. Indeed, $Y$ is solution to an infinite dimensional stochastic evolution equation driven by a trace-class noise, thus trajectories of $Y$ are $\frac12^-$ H\"older continuous. 
This permits to benefit from standard techniques that have been extensively studied in recent years for finite dimensional SDEs, in particular for the design of high order schemes for the invariant distribution (non-Markovian schemes \cite{Leimkuhler13rco}, postprocessing \cite{Abdulle17oes,Vilmart15pif}, modified equations \cite{Abdulle12hwo, Bronasco22cef, Laurent20eab}), and for the design of variance reduction techniques  (Multilevel Monte-Carlo methods (MLMC) \cite{Giles08mmc,Giles15mmc}, perturbations \cite{Abdulle19act, Duncan16vru, Lelievre13onr}).

While stochastic Runge-Kutta methods are a popular approach for the design of high-order weak methods for SDEs \cite{Burrage00oco}, the use of postprocessing techniques \cite{Brehier_Vilmart:16,Vilmart15pif} yields cheap efficient methods with high-order specifically for the invariant distribution, i.e. they can be used to increase the order $q$ without modifying the weak order $p$.
Given an integrator \eqref{equation:one_step_integrator}, a postprocessor is a random mapping $\overline{\psi}^{\Delta t}:H\to H$, which defines a perturbation $\overline{Y}_n^{\Delta t}$ of the integrator output $Y_n^{\Delta t}$, such that $(Y_n^{\Delta t},\overline{Y}_n^{\Delta t})$ is a Markov chain and
\[
\overline{Y}_n^{\Delta t}=\overline{\psi}^{\Delta t}(Y_n^{\Delta t}).
\]
With a good choice of postprocessor, for finite dimensional ergodic SDEs it is shown in \cite{Vilmart15pif} that the order for the invariant distribution of $\overline{Y}_n^{\Delta t}$ can be higher than the one of $Y_n^{\Delta t}$.
Moreover, as the postprocessor correction is only applied at the last step of the method (hence the name postprocessing), the accuracy is improved with negligible additional computational cost.

In this work, we propose in particular the following new postprocessed method where $\overline{Y}_n$ yields second order for the invariant distribution $\mu_\star$ of the original dynamics \eqref{eq:semilin}:
\begin{align*}
Y_{n+1}&=Y_n+\Delta tG(Y_n+\frac12\Delta W_n^Q)+\Delta W_n^Q,\\
\overline{Y}_n&=Y_n+\frac12\Delta W_n^Q,
\end{align*}
where $G(y)=-y+QF(y)$.
The above method is an extension of the second order Leimkuhler--Matthews method \cite{Leimkuhler13rco, Leimkuhler14otl} for preconditioned stochastic partial differential equations. That method has been first introduced in the finite dimensional setting in the context of overdamped Langevin dynamics. It has originally been proposed as a non-Markovian method and has then been reformulated as a postprocessed method in \cite{Vilmart15pif}.
It was also recently extended for nonlinear diffusion coefficient in \cite{Bronasco2025PhD, Bronasco25els}. Compared with the processed integrator introduced in~\cite{Brehier_Vilmart:16}, where no preconditioning is employed, one is able to provide rigorous justification for achieving the higher order $q=2$ for the approximation of the invariant distribution.

Let us present the main results of this article. Section~\ref{sec:precond} presents the application of the preconditioning technique in the context of stochastic evolution equations. Proposition~\ref{propo:invariant} shows that the preconditioning does not modify the invariant distribution in the considered setting (see also Lemma~\ref{propo:invariant-general} for a more general version). In the error analysis, one employs solutions to auxiliary Kolmogorov and Poisson equations, Proposition~\ref{propo:Kolmogorov} provides some regularity properties of solutions to Kolmogorov equations. In Section~\ref{section:simple_CV_analysis}, we consider the standard explicit Euler scheme
\[
Y_{n+1}=Y_n+\Delta tG(Y_n)+\Delta W_n^Q
\]
applied to the preconditioned stochastic evolution equation, but the postprocessing technique has not been applied yet. We show for that method that $q=1$, see Theorem~\ref{theo:order1}. In fact, weak error estimates with order $p=1$, which are uniform with respect to time $T$, are established, and one obtains $q=p=1$ letting $T$ go to $\infty$. Note that obtaining methods of order $q=1$ is already an improvement compared with methods applied to the original dynamics, which only have order $\frac{1}{2}-$. Next, in Section~\ref{section:high_order_inv_meas}, we show that the method proposed above, which combines the preconditioning and the postprocessing techniques, achieves order $q=2$, see Theorem~\ref{theorem:high_order_inv_measure}. We first exhibit order conditions for a more general class of methods, and then show that they are satisfied for the considered scheme. Finally, numerical experiments are presented in Section~\ref{section:num}, in order to illustrate the main results and compare different methods.

\section{Setting}\label{sec:setting}

We present in this section the standard tools for the numerical analysis of SPDEs and refer to the monographs \cite{DPZ, JentzenKloeden:11, Kruse:14, Lord_Powell_Shardlow:14} for further details.
Let $H$ be a separable, infinite dimensional, Hilbert space.  The inner product and the norm are denoted by $\langle\cdot,\cdot\rangle$ and $|\cdot|$ respectively. The space of bounded linear operators from $H$ to $H$ is denoted by $\mathcal{L}(H)$, and the associated norm is denoted by $\|\cdot\|_{\mathcal{L}(H)}$.
The following standard convention is used in this article. Let $\varphi:H\to\R$ be a function of class $\mathcal{C}^2$. For all $x\in H$, the first-order derivative $D\varphi(x)$ and the second-order derivative $D^2\varphi(x)$ are interpreted as elements of $H$ and of $\mathcal{L}(H)$ respectively, owing to the Riesz theorem.

\begin{defi}\label{defi:QAnu}
Let $\bigl(e_k\bigr)_{k\in\N}$ be a complete orthonormal system of $H$ and $\bigl(q_k\bigr)_{k\in\N}$ be a sequence of positive real numbers, such that $\sum_{k=1}^{\infty}q_k<\infty$.
Set
\begin{align*}
Qx&=\sum_{k\in\N}q_k\langle x,e_k\rangle e_k~,\quad x\in H\\
-Ax&=\sum_{k\in\N}\lambda_k\langle x,e_k\rangle e_k~,\quad x\in D(A),
\end{align*}
with $\lambda_k=q_k^{-1}$ for all $k\in\N$, and $D(A)=\left\{x\in H;~\sum_{k=1}^{\infty}\lambda_k^2\langle x,e_k\rangle^2<\infty\right\}$.
Finally, let $\nu=\mathcal{N}(0,\frac{1}{2}Q)$ be the centered Gaussian distribution on $H$, with covariance operator $\frac12 Q$.
\end{defi}

Note that $Q$ is a self-adjoint, nonnegative and trace-class (${\rm Tr}(Q)=\sum_{k=1}^{\infty}\langle Qe_k,e_k\rangle=\sum_{k=1}^{\infty}q_k<\infty$), thus $\nu$ is a well-defined Gaussian distribution on $H$. Recall that $\nu$ is the distribution of the $H$-valued random variable $\sum_{k\in\N}\frac{\sqrt{q_k}}{\sqrt{2}}\xi_ke_k$, where $\bigl(\xi_k\bigr)_{k\in\N}$ are independent standard real-valued Gaussian random variables, $\xi_k\sim\mathcal{N}(0,1)$.
Without loss of generality, assume that the sequence $\bigl(q_k\bigr)_{k\in\N}$ is non-increasing. The sequence $\bigl(\lambda_k\bigr)_{k\in\N}$ is then non-decreasing, and $\lambda_k\underset{k\to\infty}\to \infty$. Observe that $\|Q\|_{\mathcal{L}(H)}=q_1=\lambda_1^{-1}$.
For example, let $H\in L^2(0,1)$ and for all $k\in\N$, set $e_k(\cdot)=\sqrt{2}\sin\bigl(k\pi\cdot)$ and $\lambda_k=\pi^2k^2$. Then, the operator $A$ given in Definition~\ref{defi:QAnu} is the Laplace operator with homogeneous Dirichlet boundary conditions, as described in \eqref{eq:heat1d}.

For $\alpha\in(0,1)$, we define the unbounded linear operator $(-A)^{\alpha}$ as follows:
\[
(-A)^\alpha x=\sum_{k\in\N}\langle x,e_k\rangle \lambda_k^\alpha e_k,
\]
for all $x\in D((-A)^\alpha)=\left\{x\in H;~\sum_{k\in\N}\lambda_k^{2\alpha}\langle x,e_k\rangle^2<\infty\right\}$. These operators are helpful in quantifying the spatial regularity of the stochastic processes defined below.
Let
\begin{equation}\label{eq:alphasup}
\alpha_{\sup}=\sup\Big\{\alpha;{\rm Tr}\bigl((-A)^{2\alpha}Q\bigr)=\sum_{k=1}^{\infty} q_k^{1-2\alpha}<\infty\Big\}.
\end{equation}
We assume that $\alpha_{\sup}>0$.
In particular, $\alpha_{\sup}=1/4$ for the Laplace operator with homogeneous Dirichlet boundary conditions. Note that $\alpha_{\sup}\leq 1/2$.

Let $\bigl(\beta_k\bigr)_{k\in\N}$ be a sequence of independent standard real-valued Wiener processes, defined on a probability space $(\Omega,\mathcal{F},\mathbb{P})$ which satisfies the usual conditions. For all $t\ge 0$, set
\begin{align*}
W(t)&=\sum_{k\in\N}\beta_k(t)e_k,\\
W^Q(t)&=\sum_{k\in\N}\sqrt{q_k}\beta_k(t)e_k.
\end{align*}
On the one hand, $\bigl(W(t)\bigr)_{t\ge 0}$ is a cylindrical Wiener process and does not take values in $H$. On the other hand, $\bigl(W^Q(t)\bigr)_{t\ge 0}$ is a $Q$-Wiener process and takes values in $H$.
It is straightforward to check that $\nu$ is the unique invariant distribution of the $H$-valued Ornstein-Uhlenbeck processes $X^{\rm OU}$ and $Y^{\rm OU}$ given by
\[
X^{\rm OU}(t)=e^{tA}X^{\rm OU}(0)+\int_{0}^{t}e^{(t-s)A}dW(s)~,\quad Y^{\rm OU}(t)=e^{-t}Y^{\rm OU}(0)+\int_{0}^{t}e^{s-t}dW^Q(s),
\]
where $e^{tA}x=\sum_{k\in\N}e^{-t\lambda_k}\langle x,e_k\rangle e_k$ for all $t\ge 0$. Note that the stochastic convolution is well-defined with values in $H$:
\[
\E|\int_{0}^{t}e^{(t-s)A}dW(s)|^2=\sum_{k=1}^{\infty}\E|\int_{0}^te^{-\lambda_k(t-s)}d\beta_k(s)|^2=\sum_{k=1}^{\infty}\frac{1-e^{-2\lambda_kt}}{2\lambda_k}<\infty.
\]
The Ornstein-Uhlenbeck processes $X^{\rm OU}$ and $Y^{\rm OU}$ are solutions of the stochastic evolution equations, driven by additive noise,
\[
dX^{\rm OU}(t)=AX^{\rm OU}(t)dt+dW(t)~,\quad dY^{\rm OU}(t)=-Y^{\rm OU}(t)dt+dW^Q(t),
\]
respectively.

The main question studied in this article is the estimation of $\int \varphi d\mu_\star$, using numerical integrators for associated ergodic dynamics, where the probability distribution $\mu_\star$ is given below.
\begin{defi}\label{defi:mu}
Let $V:H\to\R$ be a function of class $\mathcal{C}^2$ with bounded second order derivatives.
We define
\[
d\mu_\star(x)=Z^{-1}\exp\bigl(-2V(x)\bigr)d\nu(x),
\]
where $\nu=\mathcal{N}(0,\frac12 Q)$, and $Z=\int_{H}e^{-2V(x)}d\nu(x)$.
\end{defi}

Observe that $V$ is globally Lipschitz continuous, thus its growth is at most linear. By the Fernique theorem, $Z$ is finite and thus $\mu_\star$ is well-defined. More generally, it has finite moments of all orders, precisely, for all $\kappa\in\N$,
\[
\int |x|^\kappa d\mu_\star(x)<\infty.
\]
We then define the nonlinear operators $F,G:H\to H$ as follows:
\[
F(x)=-DV(x)~,\quad G(x)=-x+QF(x).
\]

To prove ergodicity and analyze the speed of convergence to equilibrium for the continuous and discrete-time processes studied below, the following assumption gives a convenient sufficient condition. We mention that this assumption could be weakened for the study of a nonglobally Lipschitz nonlinearity $F$ as studied in \cite{BG25}, but this is out of the scope of the present paper.
\begin{ass}\label{ass:AF}
Assume that
\[
{\rm Lip}(F)=\underset{x_1\neq x_2\in H}\sup~\frac{|F(x_2)-F(x_1)|}{|x_2-x_1|}<q_1^{-1}=\lambda_1=\underset{k\in\N}\min~\lambda_k.
\]
\end{ass}

For example, let $f:\R\to\R$ be a bounded, Lipschitz continuous mapping and $H=L^2(0,1)$. The Nemytskii operator $F:H\to H$ is defined for all $x\in H$ by
\begin{equation}
\label{equation:Nemytskii}
F(x)(\cdot)=f(x(\cdot)).
\end{equation}
Then $F(x)=-DV(x)$ satisfies Assumption \ref{ass:AF}, where
\[
V(x)=\int_{0}^{1}U(x(z))dz, \quad U'=-f.
\]

To conclude this section, let us introduce a finite dimensional discretization in space, which shall prove useful for the analysis given that a number of objects are not well-posed directly in the infinite-dimensional Hilbert space $H$. Let spectral Galerkin projection operators $\bigl(\pi^K\bigr)_{K\in\N}$ be defined as follows:
\begin{equation} \label{eq:defpiK}
\pi^Kx=\sum_{k=1}^{K}\langle x,e_k\rangle e_k.
\end{equation}
Let $H^K={\rm span}\{e_1,\ldots,e_K\}$ denote the range of $\pi^K$, with $\dim H^K=K$.
We define the probability distribution $\mu_\star^K$ on $H^K$ by
\begin{equation} \label{eq:defmustarK}
d\mu_\star^K(x)=\frac{1}{Z^K}\exp\bigl(-2V(x)\bigr)d\nu^K(x),
\end{equation}
where $\nu^K=\mathcal{N}(0,\frac12 Q^K)$, with $Q^K=\pi^KQ\pi^K=\sum_{k=1}^{K}\frac{\langle \cdot,e_k\rangle}{\lambda_k}e_k$, and $Z^K=\int_{H^K}e^{-2V(x)}d\nu^K(x)$. Note that $\nu^K$ is the push-forward measure of $\nu$ by $\pi_K$.
Observe that, for any smooth test function $\varphi:H\to\R$, one has
\[
\int_{H^K}\varphi\circ \pi^K d\mu_\star^K=\frac{\int_{H^K}\varphi\circ \pi^K e^{-V}d\nu^K}{\int_{H^K}e^{-V}d\nu^K}=\frac{\int_{H}\varphi\circ \pi^K e^{-V\circ \pi^K}d\nu}{\int_{H}e^{-V\circ \pi^K}d\nu}\underset{K\to\infty}\to \int_{H}\varphi d\mu_\star.
\]
In addition, let $V^K(x)=V(x)$ for all $x\in H^K$. Then, one has the identities $-DV^K(x)=-\pi^K DV(\pi^Kx)=\pi^K F(\pi^Kx)=:F^K(x)$ for all $x\in H^K$.
Note also that
\[
\underset{K\in\N}\sup~\int_{H^K}|x^K| d\mu_\star^K(x_K)<\infty, \quad \int_{H}|x| d\mu_\star(x)<\infty.
\]

\section{Preconditioning of ergodic SPDEs}
\label{sec:precond}

In this section, we introduce the preconditioned dynamics that will permit to derive high-order methods in spite of the low-regularity of the original dynamics.
A crucial ingredient in the analysis is to consider dynamics projected on finite dimensional spaces $H^K$ with finite dimension $K$ arbitrarily large and to derive estimates with uniform constants w.r.t.\ts $K$. In particular, in the remaining of the paper, $C$ denotes a generic constant independent of $K$.

Let $\mathcal{P}$ be a bounded, self-adjoint, linear operator that commutes with $A$. This means that $\mathcal{P}x=\sum_{k\in\N}p_k\langle x,e_k\rangle e_k$, and we assume that $p_k>0$ for all $k\in\N$, and $\underset{k\in\N}\sup~p_k<\infty$. Let $\bigl(Y^{\mathcal{P}}(t)\bigr)_{t\ge 0}$ be defined as the solution of
\begin{equation}\label{eq:precondSPDE-general}
dY^{\mathcal{P}}(t)=\mathcal{P}\Bigl(AY^{\mathcal{P}}(t)+F(Y^{\mathcal{P}}(t))\Bigr)dt+\mathcal{P}^{\frac12}dW(t).
\end{equation}
For $K\in\N$, we denote $\bigl(Y^{\mathcal{P},K}\bigr)_{t\ge 0}$ the solution of
spectral Galerkin method of~\eqref{eq:precondSPDE-general},
\begin{equation}\label{eq:precondSPDE-general-Galerkin}
dY^{\mathcal{P},K}(t)=\mathcal{P}\Bigl(AY^{\mathcal{P},K}(t)+F^K(Y^{\mathcal{P},K}(t))\Bigr)dt+ dW^{\mathcal{P},K}(t), \quad W^{\mathcal{P},K}=\pi^K W^{\mathcal{P}},
\end{equation}
where $F^K=\pi^K \circ F \circ \pi^K$ and $Y^{\mathcal{P},K}(0)=\pi^K Y^{\mathcal{P}}(0)$. Observe that the projection operator $\pi^K$ defined in \eqref{eq:defpiK} commutes with $A$ and $\mathcal{P}$.
Note that choosing $\mathcal{P}=I$, problem~\eqref{eq:precondSPDE-general} then reduces to the original problem~\eqref{eq:semilin} without any preconditioning. Its spectral Galerkin projection is
\begin{equation}\label{eq:semilin_K}
dX^K(t)=AX^K(t)dt+F^K(X^K(t))dt+dW^K(t), \quad X^K(0)=\pi^K X(0).
\end{equation}

\subsection{Preconditioning preserves the invariant distribution}

A natural choice of preconditionning is $\mathcal{P}=Q=(-A)^{-1}$ which yields the modified problems
\begin{align}
\label{eq:precondSPDE}
dY(t)&=\big(-Y(t)+QF(Y(t))\big)dt+dW^{Q}(t),\\
\label{eq:precondSPDE_K}
dY^K(t)&=\big(-Y^K(t)+QF^K(Y^K(t))\big)dt+dW^{Q,K}(t),\quad
Y^K(0)=\pi^K Y(0).
\end{align}
Compared with~\eqref{eq:semilin}, we observe that the stochastic evolution equation~\eqref{eq:precondSPDE} is not a parabolic stochastic PDE anymore as the unbounded diffusion operator $A$ has been removed via preconditioning. However, contrary to the cylindrical noise in~\eqref{eq:semilin}, the stochastic evolution equation~\eqref{eq:precondSPDE} is driven by a trace-class noise, since ${\rm Tr}(Q)<\infty$. Since $G=-I+QF$ and $G^K=-I+QF^K$ are Lipschitz continuous under Assumption~\ref{ass:AF}, it is straightforward to check that, for any initial condition $y_0\in H$, both equations~\eqref{eq:precondSPDE}-\eqref{eq:precondSPDE_K} admit a unique solution $\bigl(Y(t)\bigr)_{t\ge 0}$, defined at all positive times, with $Y(0)=y_0$.

\begin{propo}\label{propo:invariant}
Assume Assumption~\ref{ass:AF} and let $Y$, $Y^K$ be the solutions of the preconditioned stochastic evolution equations \eqref{eq:precondSPDE}-\eqref{eq:precondSPDE_K}.
Then, $\mu_\star$ (respectively $\mu_\star^K$) is the unique invariant distribution for~\eqref{eq:semilin} and~\eqref{eq:precondSPDE} (respectively for~\eqref{eq:semilin_K} and~\eqref{eq:precondSPDE_K}). Moreover, exponential mixing holds: there exists $C>0$ such that for all $T\ge 0$ and all globally Lipchitz $\varphi: H\rightarrow \R$, one has
\begin{align*}
\big|\E[\varphi(Y(T))]-\int\varphi d\mu_\star\big|&\le C{\rm Lip}(\varphi)e^{-(1-\frac{{\rm Lip}(F)}{\lambda_1})T}(1+\E|Y(0)|),\\
\underset{K\in\N}\sup ~
\Big|\E[\varphi(Y^K(T))]-\textstyle{\int}\varphi d\mu_\star^K\Big|
&\le C{\rm Lip}(\varphi)e^{-(1-\frac{{\rm Lip}(F)}{\lambda_1})T}(1+\E|Y(0)|).
\end{align*}
\end{propo}

\begin{remark}
The convergence speed estimate of Proposition \ref{propo:invariant} is to be compared with the estimate for the original problem \eqref{eq:semilin}, where, under Assumption~\ref{ass:AF}, the solution $X$ satisfies 
\begin{align*}
\big|\E[\varphi(X(T))-\int\varphi d\mu_\star\big|&\le C{\rm Lip}(\varphi)e^{-(\lambda_1-{\rm Lip}(F))T}(1+\E|X(0)|).
\end{align*}
\end{remark}

Under Assumption~\ref{ass:AF}, one has the following estimates: there exists $C>0$ such that
\begin{align}\label{eq:momentY}
\underset{t\ge 0}\sup~\E|Y(t)|&\le C(1+\E|Y(0)|),\\
\label{eq:momentYK}
\underset{K\in\N,t\ge 0}\sup~\E|Y^K(t)|&\le C(1+\E|Y(0)|).
\end{align}
Proposition~\ref{propo:invariant} is then a straightforward corollary of Lemma~\ref{propo:invariant-general} below.

\begin{lemma}\label{propo:invariant-general}
Let Assumption~\ref{ass:AF} be satisfied, $\mu_\star$ and $\mu_\star^K$ be given by Definition \ref{defi:mu} and equation \eqref{eq:defmustarK}, and assume that $\lambda_1\underset{k\in\N}\sup~p_k\le \underset{k\in\N}\inf~(\lambda_kp_k)$.
Then, there exist $C,\gamma>0$ such that, for every globally Lipschitz test function $\varphi:H\to\R$ and for all $t\ge 0$,
\begin{align*}
\big|\E[\varphi(Y^{\mathcal{P}}(t))]-\int \varphi d\mu_\star\big|&\le C{\rm Lip}(\varphi)e^{-\gamma t}\bigl(1+\E|Y^{\mathcal{P}}(0)|\bigr),\\
\underset{K\in\N}\sup ~\big|\E[\varphi(Y^{\mathcal{P},K}(t))]-\int \varphi d\mu_\star^K\big|&\le C{\rm Lip}(\varphi)e^{-\gamma t}\bigl(1+\E|Y^{\mathcal{P}}(0)|\bigr),
\end{align*}
and $\mu_\star$ (resp.\ts $\mu_\star^K$) is the unique invariant distribution of~\eqref{eq:precondSPDE-general} (resp.\ts \eqref{eq:precondSPDE-general-Galerkin}).
\end{lemma}

Note that the condition in Lemma~\ref{propo:invariant-general} is satisfied for $\mathcal{P}=I$ and $\mathcal{P}=(-A)^{-1}$. In the latter case, $\lambda_1\underset{k\in\N}\sup~p_k=1=\underset{k\in\N}\inf~(\lambda_kp_k)$.

\begin{proof}[Proof of Lemma~\ref{propo:invariant-general}]
Let $K\in\N$ and consider the spectral Galerkin method $\bigl(Y^{\mathcal{P},K}\bigr)_{t\ge 0}$ defined in \eqref{eq:precondSPDE-general-Galerkin}.
First, note that $\mu_\star^K$ is the unique invariant distribution of~\eqref{eq:precondSPDE-general-Galerkin} on $H^K$. Indeed, let $\tilde{Y}^{\mathcal{P},K}(t)=\mathcal{P}^{-\frac12}Y^{\mathcal{P},K}(t)$ (note that $\mathcal{P}$ is invertible on $H^K$), and observe that this auxiliary process solves the SDE (on $H^K$)
\[
d\tilde{Y}^{\mathcal{P},K}(t)=\mathcal{P}A\tilde{Y}^{\mathcal{P},K}(t)dt+\mathcal{P}^{\frac12}F^K(\mathcal{P}^{\frac12}\tilde{Y}^{\mathcal{P},K}(t))dt+dW^K(t),
\]
where $W^K(t)=\pi^KW(t)$ is a standard Wiener process in $H^K$. Note that the drift coefficient satisfies, for all $y\in H^K$, the identity
\[
\mathcal{P}Ay+\mathcal{P}^{\frac12}F^K(\mathcal{P}^{\frac12}y)=-\frac{1}{2}D\bigl(|(-A)^{\frac12}\mathcal{P}^{\frac12}(\cdot)|^2\bigr)(y)-D\tilde{V}^{\mathcal{P},K}(y)
\]
with $\tilde{V}^{\mathcal{P},K}(y)=V(\pi^K\mathcal{P}^{\frac12}y)=V^K(\mathcal{P}^{\frac12}y)$. It is a standard result that the unique invariant distribution of $\tilde{Y}^{\mathcal{P},K}$ has density (with respect to Lebesgue measure on $H^K$) which is proportional to
\[
\exp\Bigl(-|(-A)^{\frac12}\mathcal{P}^{\frac12}y|^2-2\tilde{V}^{\mathcal{P},K}(y)\Bigr).
\]
Since $Y^{\mathcal{P},K}(t)=\mathcal{P}^{\frac12}\tilde{Y}^{\mathcal{P},K}(t)$, this identity and the expression of $\tilde{V}^{\mathcal{P},K}$ then imply that the probability distribution with density (with respect to Lebesgue measure on $H^K$) which is proportional to
\[
\exp\Bigl(-|(-A)^{\frac12}y|^2-2V^{K}(y)\Bigr)
\]
is the unique invariant distribution of~\eqref{eq:precondSPDE-general}. It is straightforward to check that this probability distribution is $\mu_\star^K$.

Second, let $Y_1^{\mathcal{P},K}$ and $Y_2^{\mathcal{P},K}$ be two solutions of~\eqref{eq:precondSPDE-general-Galerkin}, driven by the same noise, and starting from two different initial conditions $y_1^K$ and $y_2^K$. Let $\delta Y^K(t)=Y_2^{\mathcal{P},K}(t)-Y_1^{\mathcal{P},K}(t)$, and $\delta F^K(t)=F^K(Y_2^{\mathcal{P},K}(t))-F^K(Y_1^{\mathcal{P},K}(t))$. Then $dR(t)=\mathcal{P}A\delta Y^K(t)dt+\mathcal{P}\delta F^K(t)$, and
\begin{align*}
\frac{1}{2}\frac{d|\delta Y^K(t)|^2}{dt}&=\langle \mathcal{P}A\delta Y^K(t),\delta Y^K(t)\rangle+\langle \mathcal{P}\delta F^K(t),\delta Y^K(t)\rangle\\
&\le -\underset{1\le k\le K}\inf~\lambda_kp_k |\delta Y^K(t)|^2+\underset{1\le k\le K}\sup~p_k |\delta F^K(t)||\delta Y^K(t)|\\
&\le \Bigl(-\underset{k\in\N}\inf~\lambda_kp_k +\underset{k\in\N}\sup~p_k {\rm Lip}(F)\Bigr)|\delta Y^K(t)|^2.
\end{align*}
Applying the Gronwall lemma yields
\[
|\delta Y^K(t)|\le e^{-\gamma(\mathcal{P},F)t}|\delta Y^K(0)|,
\]
with $\gamma(\mathcal{P},F)=\underset{k\in\N}\inf~\lambda_kp_k -\underset{k\in\N}\sup~p_k {\rm Lip}(F)>0$, for all $t\ge 0$, and all initial conditions $y_1^K$ and $y_2^K$. Since $\varphi$ is globally Lipschitz continuous, for all $y^K\in H^K$, one has
\begin{align*}
\big|\E_{y^K}[\varphi(Y^{\mathcal{P},K}(t))]-\int \varphi d\mu_\star^{K}\big|&= \big|\E_{y^K}[\varphi(Y^{\mathcal{P}}(t))]-\E_{\mu^K}[\varphi(Y^{\mathcal{P}}(t))]\big|\\
&\le {\rm Lip}(\varphi)\iint_{H^K\times H^K}\E|\delta Y^K(t)| d\mu_\star^{K}(y_2^K) d\delta_{y^K}(y_K^1)\\
&\le {\rm Lip}(\varphi)e^{-\gamma(\mathcal{P},F)t}\int_{H^K}|\cdot -y^K|d\mu_{\star}^K.
\end{align*}
Finally, taking to the limit $K\to\infty$ concludes the proof.
\end{proof}

In the sequel the preconditioning operator $\mathcal{P}$ is chosen as $\mathcal{P}=(-A)^{-1}$ (unless specified otherwise), and thus the stochastic evolution equation~\eqref{eq:semilin} is considered. This choice is motivated by a practical reason: in many situations, (a discrete version of) the operator $A$ is known, and it is not required to compute numerically its inverse $\mathcal{P}=(-A)^{-1}$ and its square root $\mathcal{P}^{\frac12}$. Instead, in the numerical integrators presented below, computing the LU and Cholesky decompositions is sufficient.
The choice $\mathcal{P}=(-A)^{-1}$ is also motivated by the following theoretical property, stated below: the trajectories of the solution $\bigl(Y(t)\bigr)_{t\ge 0}$ are almost surely $\alpha$-H\"older continuous for all $\alpha\in(0,\frac12)$. This property follows from the application of the Kolmogorov-Centzov regularity criterion, and the estimate
\[
\bigl(\E|Y(t)-Y(s)|^q\bigr)^{\frac1q}\le C|t-s|^{\frac12},~\forall s,t\in[0,T].
\]
That estimate is established using the trace-class property of the noise (${\rm Tr}(Q)<\infty$). The computations are straightforward and are thus left for the readers.

\begin{remark}
On the one hand, this regularity is much better than for trajectories of the original SPDE~\eqref{eq:semilin}, which are only $\frac{\alpha}{2}$-H\"older continuous for such parameters $\alpha\in(0,\frac12)$. On the other hand, this temporal regularity is the same as for solutions of finite dimensional SDEs. As a consequence, it is natural to expect that the order of convergence with respect to the time-step size is increased, and that one can design higher order methods.
Note that the spatial regularity of trajectories is unchanged by the preconditioning procedure. This is consistent with the fact that $\bigl(X(t)\bigr)_{t\ge 0}$ and $\bigl(Y(t)\bigr)_{t\ge 0}$ have the same invariant distribution $\mu_\star$. Note also that applying the spectral Galerkin discretization to the preconditioned equation~\eqref{eq:precondSPDE} yields the same process as when applying the preconditioning technique after the spectral Galerkin discretization of the original equation~\eqref{eq:precondSPDE}. As a consequence, applying the preconditioning technique does not modify the error due to spatial discretization.
\end{remark}

\subsection{Kolmogorov and Poisson equations for the preconditioned dynamics}

Let $M>0$, $\beta\in[0,2\alpha_{\sup})$ (defined in \eqref{eq:alphasup}), $H^\beta=D((-A)^{\beta/2})$, and $|x|_\beta=|(-A)^{\beta/2}x|$, we say that $F\colon H\rightarrow H$ is in $\mathfrak{X}^{M,\beta}_P$ if for all $K$, $F^K\colon H^{K,\beta}\rightarrow H^K$ is $M$ times differentiable and if there exists $C>0$, $\kappa>0$ independent of $K$ such that for all $1< m\leq M$,
\[
|D^m F^K(x)(h_1,\dots,h_m)|\leq C(1+|x|_\beta^\kappa)|h_1|_\beta\dots |h_m|_\beta,\quad x,h_1,\dots,h_m\in H^\beta.
\]
Similarly, we say that $\varphi\colon H\rightarrow \R$ is in $\mathcal{C}^{M,\beta}_P$ if for all $K$, $\varphi^K\colon H^{K,\beta}\rightarrow \R$ is $M$ times differentiable and if there exists $C>0$, $\kappa>0$ independent of $K$ such that for all $1\leq m\leq M$,
\[
|D^m \varphi^K(x)(h_1,\dots,h_m)|\leq C(1+|x|_\beta^\kappa)|h_1|_\beta\dots |h_m|_\beta,\quad x,h_1,\dots,h_m\in H^\beta.
\]
For instance for the one-dimensional Laplacian, observe that the Nemytskii operator \eqref{equation:Nemytskii} for $f(x)=-x+\cos(x)$ satisfies $F\in \mathfrak{X}^{M,\beta}_P$ for $M$ arbitrarily large and some $\beta\in (0,1/4)$.

Define the auxiliary functions $v$ and $\Psi$ as follows, where for the sake of brevity, we omit the index $K$ of the spectral Galerkin projection in the notation: for all $t\ge 0$ and $y\in H^K$,
\begin{equation}\label{eq:v}
v(t,y)=\E_y[\varphi(Y^K(t))],
\end{equation}
and
\begin{equation} \label{eq:defPsi}
\Psi(y)=\int_{0}^{\infty}\bigl(v(t,y)-\int\varphi d\mu_\star^K\bigr)dt.
\end{equation}
Note that the function $\Psi$ is well-defined owing to Proposition~\ref{propo:invariant}.

Introduce the infinitesimal generator $\mathcal{L}$ associated with the preconditioned dynamics~\eqref{eq:precondSPDE}: if $\varphi:H\to\R$ is twice differentiable,
\[
\mathcal{L}\varphi(y)=D\varphi(y).G^K(y)
+\sum_{k=1}^K q_kD^2\varphi(y).(e_k,e_k).
\]
On the one hand, $v$ is the solution of the Kolmogorov equation
\begin{equation}\label{eq:Kolmogorov}
\frac{\partial}{\partial t}v(t,y)=\mathcal{L}v(t,y)= 
Dv(t,y).G^K(y)+\frac12 \sum_{k=1}^K q_kD^2v.(e_k,e_k),
\end{equation}
with initial condition $v(0,\cdot)=\varphi$, where $Dv(t,y)$ and $D^2v(t,y)$ denote spatial derivatives of $v$.
On the other hand, $\Psi$ is the solution of the Poisson equation with $x\in H^K$,
\begin{equation}\label{eq:Poisson}
-\mathcal{L}\Psi(x)=\varphi(x)-\int\varphi d\mu_\star^K.
\end{equation}

The first and second order spatial derivatives of $v$ satisfy the estimates given by Proposition~\ref{propo:Kolmogorov} below.
\begin{propo}\label{propo:Kolmogorov}
Assume Assumption~\ref{ass:AF} with $\gamma=1-\frac{{\rm Lip}(F)}{\lambda_1}>0$. Consider $F\in \mathfrak{X}^{M,\beta}_P$, $v(t,y)$ defined in \eqref{eq:v}, and $\varphi\in \mathcal{C}^{M,\beta}_P$.
Then, there exists $C>0$, such that for all $t\ge 0$ and all $y\in H$, $h_1,\dots,h_m\in H^\beta$, one has
\begin{equation}\label{eq:propo-Kolmogorov}
|D^m v(t,y)(h_1,\dots,h_m)|\leq C (1+|y|_\beta^\kappa) e^{-\gamma t}|h_1|_\beta\dots |h_m|_\beta.
\end{equation}
\end{propo}

Note that the estimates~\eqref{eq:propo-Kolmogorov} justify that $\Psi$ is of class $\mathcal{C}^2$ and $\Psi$ is solution of the Poisson equation~\eqref{eq:Poisson}.

\begin{proof}
For the sake of clarity, we present the proof of \eqref{eq:propo-Kolmogorov} for $M=1,2,3$. The estimates are proven analogously for higher $M$.
\smallskip

\noindent
Case $M=1$:
One has
\[
Dv(t,y).h=\E_y[D\varphi(Y^K(t)).\eta^h(t)]
\]
with
\[
\frac{d\eta^h(t)}{dt}=-\eta^h(t)+QDF^K(Y^K(t)).\eta^h(t),\quad \eta^h(0)=h.
\]
By the assumption on $\varphi$:
\[
|Dv(t,y).h|\le C\E[|(-A)^{\beta}\eta^h(t)|].
\]
Decompose
\[
\tilde{\eta}^h(t)=\eta^h(t)-e^{-t}h,
\]
with
\[
\frac12\frac{d\tilde{\eta}^h(t)}{dt}=-\tilde{\eta}^h(t)+QDF^K(Y^K(t)).\tilde{\eta}^h(t)+QDF^K(Y^K(t)).e^{-t}h.
\]
Using an energy estimate in the norm $|(-A)^{\frac12}\cdot|=|Q^{-\frac12}\cdot|$, one obtains
\begin{align*}
\frac12\frac{d|(-A)^{\frac12}\eta^h(t)|^2}{dt}&=-|(-A)^{\frac12}\tilde{\eta}^h(t)|^2+\langle DF^K(Y^K(t)).\tilde{\eta}^h(t),\tilde{\eta}^h(t)\rangle\\
&+\langle DF^K(Y^K(t)).e^{-t}h,\tilde{\eta}^h(t)\rangle\\
&\le -|(-A)^{\frac12}\tilde{\eta}^h(t)|^2+{\rm Lip}(F)|\tilde{\eta}^h(t)|^2+\epsilon|\tilde{\eta}^h(t)|^2+\frac{1}{4\epsilon}e^{-2t}|DF^K(Y^K(t)).h|^2\\
&\le -(\gamma-\epsilon)|(-A)^{\frac12}\tilde{\eta}^h(t)|^2+C_\epsilon e^{-2t}|h|^2,
\end{align*}
where $\epsilon>0$ can be chosen arbitrarily small.

Using the Gronwall inequality (observe that $\tilde{\eta}^h(0)=0$) one obtains
\[
|(-A)^{\frac12}\tilde{\eta}^h(t)|^2\le C_\epsilon \int_0^t e^{-2(\gamma-\epsilon)(t-s)}e^{-2s}ds|h|^2\le C_\epsilon e^{-2(\gamma-\epsilon)t}|h|^2.
\]
Using the upper bound $\beta\le \frac12$, and recalling that $\eta^h(t)=e^{-t}h+\tilde{\eta}^h(t)$, one obtains
\begin{align*}
|(-A)^{\beta}\eta^h(t)|&\le e^{-t}|(-A)^{\beta}h|+C|(-A)^{\frac12}\tilde{\eta}^h(t)|\\
&\le e^{-t}|(-A)^{\beta}h|+Ce^{-(\gamma-\epsilon)t}|h|\\
&\le Ce^{-(\gamma-\epsilon)t}|(-A)^{\beta}h|,
\end{align*}
which is the same as above (up to $\gamma$ being replaced by $\gamma-\epsilon$).
We then obtain
\[
|Dv(t,y).h|\le Ce^{-\gamma t}|(-A)^{\beta}h|.
\]
\smallskip

\noindent
Case $M=2$:
One has
\[
D^2v(t,y).(h_1,h_2)=\E_y[D\varphi(Y^K(t)).\zeta^{h_1,h_2}(t)]+\E_y[D^2\varphi(Y^K(t)).(\eta^{h_1}(t),\eta^{h_2}(t))],
\]
with $\zeta^{h_1,h_2}(0)=0$ and
\[
\frac{d\zeta^{h_1,h_2}(t)}{dt}=-\zeta^{h_1,h_2}(t)+QDF^K(Y^K(t)).\zeta^{h_1,h_2}(t)+QD^2F^K(Y^K(t)).(\eta^{h_1}(t),\eta^{h_2}(t)).
\]
Using the assumption on $\varphi$ and the upper bound obtained in the previous case $M=1$:
\[
\big|\E[D^2\varphi(Y^K(t)).(\eta^{h_1}(t),\eta^{h_2}(t))]\big|\le Ce^{-2\gamma t}|(-A)^{\beta}h_1||(-A)^{\beta}h_2|.
\]
To prove upper bounds for $|(-A)^{\beta}\zeta^{h_1,h_2}(t)|$, the second strategy described in the case $M=1$ is employed:
\begin{align*}
\frac12\frac{d|(-A)^{\frac12}\zeta^{h_1,h_2}(t)|^2}{dt}&=-|(-A)^{\frac12}\zeta^{h_1,h_2}(t)|^2+\langle DF^K(Y^K(t)).\zeta^{h_1,h_2}(t),\zeta^{h_1,h_2}(t)\rangle\\
&\quad+\langle D^2F^K(Y^K(t)).(\eta^{h_1}(t),\eta^{h_2}(t)),\zeta^{h_1,h_2}(t)\rangle\\
&\le -\gamma|(-A)^{\frac12}\zeta^{h_1,h_2}(t)|^2+|D^2F^K(Y^K(t)).(\eta^{h_1}(t),\eta^{h_2}(t))||\zeta^{h_1,h_2}(t)|\\
&\le -(\gamma-\epsilon)|(-A)^{\frac12}\zeta^{h_1,h_2}(t)|^2+C_\epsilon|(-A)^{\beta}\eta^{h_1}(t)||(-A)^{\beta}\eta^{h_2}(t)|\\
&\le -(\gamma-\epsilon)|(-A)^{\frac12}\zeta^{h_1,h_2}(t)|^2+C_\epsilon e^{-2\gamma t}|(-A)^{\beta}h_1|^2|(-A)^{\beta}h_2|^2.
\end{align*}
Applying the Gronwall inequality gives for all $t\ge 0$
\[
|(-A)^{\frac12}\zeta^{h_1,h_2}(t)|^2\le C_\epsilon e^{-2(\gamma-\epsilon)t}|(-A)^{\beta}h_1|^2|(-A)^{\beta}h_2|^2.
\]
In addition, one has $\beta\le \frac12$, therefore one obtains
\[
|(-A)^{\beta}\zeta^{h_1,h_2}(t)|\le C_\epsilon e^{-(\gamma-\epsilon)t}|(-A)^{\beta}h_1||(-A)^{\beta}h_2|.
\]
Finally one obtains
\[
\big|D^2v(t,y).(h_1,h_2)\big|\le Ce^{-\gamma t}|(-A)^{\beta}h_1||(-A)^{\beta}h_2|.
\]
\smallskip

\noindent
Case $M=3$:
One has
\begin{align*}
D^3v(t,y).(h_1,h_2,h_3)&=\E_y[D\varphi(Y^K(t)).\xi^{h_1,h_2,h_3}(t)]\\
&+\E_y[D^2\varphi(Y^K(t)).(\eta^{h_1}(t),\zeta^{h_2,h_3}(t))]\\
&+\E_y[D^2\varphi(Y^K(t)).(\eta^{h_2}(t),\zeta^{h_3,h_1}(t))]\\
&+\E_y[D^2\varphi(Y^K(t)).(\eta^{h_3}(t),\zeta^{h_1,h_2}(t))]\\
&+\E_y[D^3\varphi(Y^K(t)).(\eta^{h_1}(t),\eta^{h_2}(t),\eta^{h_3}(t))]
\end{align*}
with $\xi^{h_1,h_2,h_3}(0)=0$ and
\begin{align*}
\frac{d\xi^{h_1,h_2,h_3}(t)}{dt}&=-\xi^{h_1,h_2,h_3}(t)+QDF^K(Y^K(t)).\xi^{h_1,h_2,h_3}(t)\\
&+QD^2F^K(Y^K(t)).(\eta^{h_1}(t),\zeta^{h_2,h_3}(t))\\
&+QD^2F^K(Y^K(t)).(\eta^{h_2}(t),\zeta^{h_3,h_1}(t))\\
&+QD^2F^K(Y^K(t)).(\eta^{h_3}(t),\zeta^{h_1,h_2}(t))\\
&+QD^3F^K(Y^K(t)).(\eta^{h_1}(t),\eta^{h_2}(t),\eta^{h_3}(t)).
\end{align*}
It is only necessary to deal with the term $\E_y[D\varphi(Y^K(t)).\xi^{h_1,h_2,h_3}(t)]$, since all the other terms in the expression of $D^3v(t,y).(h_1,h_2,h_3)$ are upper bounded by $Ce^{-\gamma t}|(-A)^{\beta}h_1||(-A)^{\beta}h_2||(-A)^{\beta}h_3|$.

Using the same strategy as above, one obtains
\begin{align*}
\frac12\frac{d|(-A)^{\frac12}\xi^{h_1,h_2,h_3}(t)|^2}{dt}&\le -|(-A)^{\frac12}\xi^{h_1,h_2,h_3}(t)|^2+Ce^{-3\gamma t}|(-A)^{\beta}h_1||(-A)^{\beta}h_2||(-A)^{\beta}h_3|,
\end{align*}
and applying the Gronwall inequality gives for all $t\ge 0$
\[
|(-A)^{\frac12}\xi^{h_1,h_2,h_3}(t)|^2\le C_\epsilon e^{-2(\gamma-\epsilon)t}|(-A)^{\beta}h_1|^2|(-A)^{\beta}h_2|^2|(-A)^{\beta}h_3|^2.
\]
In addition, one has $\beta\le \frac12$, therefore one obtains
\[
|(-A)^{\beta}\xi^{h_1,h_2,h_3}(t)|\le C_\epsilon e^{-(\gamma-\epsilon)t}|(-A)^{\beta}h_1||(-A)^{\beta}h_2||(-A)^{\beta}h_3|.
\]
Finally, we obtain the estimate
\[
\big|D^3v(t,y).(h_1,h_2,h_3)\big|\le Ce^{-\gamma t}|(-A)^{\beta}h_1||(-A)^{\beta}h_2||(-A)^{\beta}h_3|,
\]
which concludes the proof.
\end{proof}

\begin{cor}
Under the assumptions of Proposition \ref{propo:Kolmogorov} with $\beta=0$, for all $t_2\ge t_1\ge 0$ and all $y,h\in H$, one has
\begin{equation}\label{eq:propo-Kolmogorov2}
|Dv(t_2,y).h-Dv(t_1,y).h|\le C (1+|y|)\sqrt{t_2-t_1}e^{-\gamma' t_1}|h|.
\end{equation}
\end{cor}

\begin{proof}
It remains to prove~\eqref{eq:propo-Kolmogorov2}. Let $y,h\in H$ and $0\le t_1\le t_2$, then one has
\begin{align*}
Dv(t_2,y).h-Dv(t_1,y).h&=\E_y\bigl[\bigl(D\varphi(Y^K(t_2))-D\varphi(Y^K(t_1))\bigr).\eta^h(t_2)\bigr]\\
&~+\E_y\bigl[D\varphi(Y^K(t_1)).\bigl(\eta^h(t_2)-\eta^h(t_1)\bigr)\bigr].
\end{align*}
On the one hand, since by assumption $\varphi$ is twice differentiable, one has
\begin{align*}
\big|\E_y\bigl[\bigl(D\varphi(Y^K(t_2))-D\varphi(Y^K(t_1))\bigr).\eta^h(t_2)\bigr]\big|&\le e^{-\gamma t_2}|h|\E_y|Y^K(t_2)-Y^K(t_1)|\\
&\le \sqrt{t_2-t_1}(1+\underset{t\ge 0}\sup~\E_y[|Y^K(s)|])e^{-\gamma t_1}|h|.
\end{align*}
The last inequality follows from the identity $Y^K(t_2)-Y^K(t_1)=\int_{t_1}^{t_2}G^K(Y^K(s))ds+W^Q(t_2)-W^Q(t_1)$, from the equality $\E|W^Q(t_2)-W^Q(t_1)|^2={\rm Tr}(Q)(t_2-t_1)$, and the moment bounds~\eqref{eq:momentY} combined with the global Lipschitz continuity of $F$.

On the other hand, using the identity $\eta^h(t)=e^{-t}h+\int_0^te^{s-t}DF^K(Y^K(s)).\eta^h(s)ds$, the global Lipschitz continuity of $F$, and the Gronwall estimate $|\eta^h(t)|\le e^{-\gamma t}|h|$, one obtains, for all $0\le t_1\le t_2$,
\begin{align*}
\E_y\bigl[D\varphi(Y^K(t_1)).\bigl(\eta^h(t_2)-\eta^h(t_1)\bigr)\bigr]&\le C|\eta^h(t_2)-\eta^h(t_1)|\\
&\le\bigl(e^{-t_1}-e^{-t_2}\bigr)|h|\\
&~+\int_{t_1}^{t_2}e^{s-t_2}\big|QDF^K(Y^K(s)).\eta^h(s)\big|ds\\
&~+\int_{0}^{t_1}\bigl(e^{s-t_1}-e^{s-t_2}\bigr)\big|QDF^K(Y^K(s)).\eta^h(s)\big|ds\\
&\le e^{-t_1}(t_2-t_1)|h|+C(t_2-t_1)\underset{s\ge t_1}\sup~|\eta^h(s)|\\
&~+C(t_2-t_1)\int_{0}^{t_1}e^{-(t_1-s)}e^{-\gamma s}|h|ds\\
&\le C e^{-\gamma' t_1}(t_2-t_1)|h|.
\end{align*}
Gathering the estimates then concludes the proof of~\eqref{eq:propo-Kolmogorov2}.
\end{proof}

\section{High order one methods for sampling the invariant distribution}
\label{section:simple_CV_analysis}

In order to discretize the stochastic evolution equation \eqref{eq:precondSPDE}, the simplest numerical method is the explicit Euler scheme:
\begin{equation}\label{eq:EulerExplicite}
Y_{n+1}=Y_n+\Delta tG(Y_n)+\Delta W_n^Q,
\end{equation}
where $\Delta W_n^Q=W^Q(t_{n+1})-W^Q(t_n)$, $t_n=n\Delta t$.
More generally, we also consider the $\theta$-method, with $\theta\in[0,1]$:
\begin{equation}\label{eq:thetascheme}
Y_{n+1}=\frac{1-(1-\theta)\Delta t}{1+\theta \Delta t}Y_n+\frac{\Delta t}{1+\theta\Delta t}QF(Y_n)+\frac{1}{1+\theta\Delta t}\Delta W_n^Q.
\end{equation}
This includes the semilinear implicit Euler method for $\theta=1$, the Crank-Nicholson scheme for $\theta=\frac12$, and the explicit Euler method \eqref{eq:EulerExplicite} for $\theta=0$.

\subsection{Moment bounds and ergodicity}

The ergodicity of the numerical scheme $\bigl(Y_n\bigr)_{n\ge 0}$ defined by~\eqref{eq:EulerExplicite}, for sufficiently small time step size $\Delta t$, is ensured by Lemma~\ref{lem:ergoYn} below.
\begin{lemma}\label{lem:ergoYn}
Let Assumption~\ref{ass:AF} be satisfied.
Let $\bigl(Y_n\bigr)_{n\ge 0}$ be given by~\eqref{eq:EulerExplicite}.
Then, there exists $C>0$, such that
\begin{equation}\label{eq:moment-num-Euler}
\underset{\Delta t\in(0,1]}\sup~\underset{n\ge 0}\sup~\E|Y_n|^2\le C(1+\E|Y_0|^2).
\end{equation}
Moreover, if $\Delta t\in(0,1]$, then $\bigl(Y_n\bigr)_{n\ge 0}$ is an ergodic discrete-time Markov process, with unique invariant distribution denoted by $\mu^{\Delta t}$. One has the estimate
\[
\underset{\Delta t\in(0,1]}\sup~\int |y|^2 d\mu^{\Delta t}(y)\le C.
\]
Finally, ifs $\gamma(\Delta t)=1-\Delta t(1-q_1{\rm Lip}(F))\in(0,1)$, then for any globally Lipschitz test function $\varphi:H\to\R$ and all $n\in\N$, one has
\[
\big|\E[\varphi(Y_n)]-\int\varphi d\mu^{\Delta t} \big|\le {\rm Lip}(\varphi)\gamma(\Delta t)^n(1+\E[|Y_0|]).
\]
\end{lemma}

\begin{proof}[Proof of Lemma~\ref{lem:ergoYn}]
To prove the moment bounds~\eqref{eq:moment-num-Euler}, introduce the auxiliary processes given by
\[
Z_{n+1}=(1-\Delta t)Z_n+\Delta W_n^Q~,\quad Z_0=0~,\quad R_n=Y_n-Z_n.
\]
First, one has
\[
\E|Z_{n+1}|^2=|1-\Delta t|^2\E|Z_n|^2+\Delta t {\rm Tr}(Q).
\]
Thus
\[
\underset{n\ge 0}\sup~\E|Z_n|^2=\Delta t{\rm Tr}(Q)\sum_{m=0}^{\infty}|1-\Delta t|^{2m}\le {\rm Tr}(Q).
\]
Second, one has
\[
R_{n+1}=(1-\Delta t) R_n+\Delta t QF(R_n+Z_n),
\]
and the global Lipschitz continuity of $F$ yields
\begin{align*}
|R_{n+1}|\le (|1-\Delta t|+\Delta tq_1{\rm Lip}(F))|R_n|+C\Delta t(1+|Z_n|).
\end{align*}
Using $\Delta t\in (0, 1]$, one has
\[
\gamma(\Delta t)=1-\Delta t(1-q_1{\rm Lip}(F)) <1.
\]
One thus obtains
\begin{align*}
\bigl(\E|R_{n}|^2\bigr)^{\frac12}&\le \gamma(\Delta t)^{n}\bigl(\E|R_0|^2\bigr)^{\frac12}+C\Delta t \sum_{m=0}^{n-1}\gamma(\Delta t)^m \underset{m\ge 0}\sup~\bigl(\E|Z_m|^2\bigr)^{\frac12}\\
&\le \bigl(\E|Y_0|^2\bigr)^{\frac12}+C\Delta t \frac{1}{1-\gamma(\Delta t)}\\
&\le \bigl(\E|Y_0|^2\bigr)^{\frac12}+C\frac{1}{1-q_1{\rm Lip}(F)}.
\end{align*}
Gathering the estimates concludes the proof of the moment bounds~\eqref{eq:moment-num-Euler}.

It remains to establish the ergodicity of the discrete-time Markov process $\bigl(Y_n\bigr)_{n\ge 0}$. Let us first check the uniqueness of invariant distributions. Let $\bigl(Y_n^{1}\bigr)_{n\in\N_0}$ and $\bigl(Y_n^{2}\bigr)_{n\in\N_0}$ be two solutions of~\eqref{eq:EulerExplicite}, driven by the same noise, and starting from two different initial conditions $y^1$ and $y^2$. Let $\delta Y_n=Y_n^2-Y_n^1$. Then, for all $n\ge 0$,
\begin{align*}
|\delta Y_{n+1}|&\le |1-\Delta t||\delta Y_n|+\Delta t\big|QF(Y_n^2)-QF(Y_n^1)\big|\\
&\le \gamma(\Delta t)|\delta Y_n|.
\end{align*}
Thus for all $n\ge 0$,
\[
|\delta Y_n|\le \gamma(\Delta t)^n|y^2-y^1|,
\]
with $\gamma(\Delta t)\in(0,1)$. This yields uniqueness of the invariant distribution.

The existence of an invariant distribution is obtained using the Krylov-Bogoliubov criterion. Indeed, let $\alpha\in(0,\alpha_{\sup})$. Then ${\rm Tr}\bigl((-A)^{2\alpha}Q\bigr)<\infty$, moreover the set $\left\{x\in H;~|(-A)^\alpha x|\le M\right\}$ is compact in $H$ for all $M>0$.
Assume that $Y_0=0$. Then, for all $n\in\N$,
\[
Y_n=\Delta t\sum_{m=0}^{n-1}(1-\Delta t)^{n-1-m}QF(Y_m)+\sum_{m=0}^{n-1}(1-\Delta t)^{n-1-m}\Delta W_m^Q.
\]
Since $(-A)^{\alpha}Q$ is a bounded linear operator ($\alpha\le 1$), using the Lipschitz continuity of $F$, the moment bounds~\eqref{eq:moment-num-Euler} and the equality $\E|(-A)^{\alpha}\Delta W_m^Q|^2=\Delta t{\rm Tr}\bigl((-A)^{2\alpha}Q\bigr)$, one obtains
\begin{align*}
\bigl(\E|(-A)^{\alpha}Y_n|^2\bigr)^{\frac12}&\le \Delta t\sum_{m=0}^{n-1}(1-\Delta t)^{n-1-m}C(1+\bigl(\E|Y_m|^2\bigr)^{\frac12})\\
&~+\Bigl({\rm Tr}\bigl((-A)^{2\alpha}Q\bigr)\Delta t\sum_{m=0}^{n-1}(1-\Delta t)^{2(n-1-m)}\Bigr)^{\frac12}\\
&\le C+{\rm Tr}\bigl((-A)^{2\alpha}Q\bigr)^{\frac12}.
\end{align*}
As a consequence, we deduce $\underset{n\in\N}\sup~\E|(-A)^\alpha Y_n|^2<\infty$ if $\alpha\in(0,\alpha_{\sup})$. Applying the Krylov-Bogoliubov criterion then concludes the proof of the existence of an invariant distribution $\mu^{\Delta t}$, and Lemma \ref{lem:ergoYn} is proved.
\end{proof}

\begin{remark}
Lemma~\ref{lem:ergoYn} remains valid for other integrators, such as the $\theta$-method \eqref{eq:thetascheme} or the method \eqref{eq:EulerExplicitePostprocessed}, assuming the timestep $\Delta t$ small enough.
Moreover, if $\theta\ge \frac12$, then there is no such timestep restriction to obtain moment bounds and ergodicity.
\end{remark}

\subsection{The Gaussian case}

In this section, we detail the Gaussian case for the $\theta$-method \eqref{eq:thetascheme}. We have the following result.
\begin{propo}
Assume that $F=0$ and that $\Delta t<\frac{2}{1-2\theta}$ if $\theta < 1/2$. Then $\bigl(Y(t)\bigr)_{t\ge 0}$ and $\bigl(Y_n\bigr)_{n\ge 0}$ are Gaussian processes with invariant distributions $\nu$, $\nu^{\Delta t,\theta}$ satisfying
\[
\big|\int \varphi d\nu^{\Delta t,\theta}-\int\varphi d\nu\big|
\leq\sup~\|D^2\varphi(y)\|_{\mathcal{L}(H)}\frac{1-\sigma_\theta(\Delta t)^2}{2}{\rm Tr}(Q),
\]
where
\[
\sigma_\theta(\Delta t)^2=\frac{2}{2+(2\theta-1)\Delta t}.
\]
In particular, for $\theta=\frac{1}{2}$, the bias for the invariant distribution vanishes.
\end{propo}

\begin{proof}
 It is straightforward to check that the numerical invariant distribution $\mu^{\Delta t,\theta}=\nu^{\Delta t,\theta}$ is a centered Gaussian distribution, with covariance $\frac{\sigma_\theta(\Delta t)^2}{2}Q$.
If $\theta=\frac12$, then $\sigma_\theta(\Delta t)^2=1$, thus $\nu$ is the invariant distribution of the discrete-time process $\bigl(Y_n\bigr)_{n\ge 0}$.

Assume now that $\theta\in[0,1]\setminus\{\frac12\}$, then one has
\[
\sigma_\theta(\Delta t)^2=1+{\rm O}(\Delta t).
\]
If $\varphi:H\to\R$ is of class $\mathcal{C}^2$, with bounded second order derivative, then
\begin{equation}\label{eq:error_Gaussian_C2}
\big|\int \varphi d\nu^{\Delta t,\theta}-\int\varphi d\nu\big|\le \frac{{\rm Tr}(Q)}{2} \underset{y\in H}\sup~\|D^2\varphi(y)\|_{\mathcal{L}(H)}\big|\sigma(\Delta t)^2-1\big|={\rm O}(\Delta t).
\end{equation}
Indeed, assume without loss of generality that $\sigma_\theta(\Delta t)^2<1$. Let $Y^{\Delta t,\theta}\sim \mathcal{N}(0,\frac{\sigma_\theta(\Delta t)^2}{2}Q)=\nu^{\Delta t,\theta}$ and $R^{\Delta t,\theta}\sim \mathcal{N}(0,\frac{1-\sigma_\theta(\Delta t)^2}{2}Q)$ be two independent Gaussian random variables. Then the sum $Y^{\Delta t,\theta}+R^{\Delta t,\theta}$ has distribution $\mathcal{N}(0,\frac12 Q)=\nu$, which yields
\begin{align*}
\big|\int \varphi d\nu^{\Delta t,\theta}-\int\varphi d\nu\big|&=\big|\E[\varphi(Y^{\Delta t,\theta})]-\E[\varphi(Y^{\Delta t,\theta}+R^{\Delta t,\theta})]\big|\\
&=\big|\E[\varphi(Y^{\Delta t,\theta})]-\E[\varphi(Y^{\Delta t,\theta}+R^{\Delta t,\theta})]-\E[\langle D\varphi(Y^{\Delta t,\theta}),R^{\Delta t,\theta}\rangle]\big|\\
&\le \sup~\|D^2\varphi(y)\|_{\mathcal{L}(H)}\E[|R^{\Delta t,\theta}|^2]\\
&=\sup~\|D^2\varphi(y)\|_{\mathcal{L}(H)}\frac{1-\sigma_\theta(\Delta t)^2}{2}{\rm Tr}(Q).
\end{align*}
The case $\sigma_\theta(\Delta t)^2>1$ is treated by a similar argument.
\end{proof}

\begin{propo}
The Gaussian distributions $\nu$ and $\nu^{\Delta t,\theta}$ are singular, as soon as $\sigma_\theta(\Delta t)^2\neq 1$. As a consequence, when $\Delta t\to 0$, the convergence does not hold in total variation distance, more precisely
\begin{equation}\label{eq:error_Gaussian_TV}
d_{\rm TV}(\nu^{\Delta t,\theta},\nu)=\frac12\underset{\varphi\in \mathcal{B}_b(H,\R),\|\varphi\|_\infty\le 1}\sup~\big|\int \varphi d\nu^{\Delta t,\theta}-\int\varphi d\nu\big|=1,
\end{equation}
where $\mathcal{B}_b(H,\R)$ denotes the space of bounded measurable functions from $H$ to $\R$.
\end{propo}

\begin{proof}
Assume without loss of generality that $\sigma(\Delta t)^2>1$.
Let $Y^{\Delta t,\theta}\sim \nu^{\Delta t,\theta}$ and $Y\sim \nu$, then by the Law of Large Numbers, almost surely
\[
\frac{1}{\sqrt{K}}\sum_{k=1}^{K}\Bigl(\frac{|\langle Y^{\Delta t,\theta},e_k\rangle|^2}{q_k}-1\Bigr)\underset{K\to\infty}\sim \sqrt{K}\frac{\sigma_\theta(\Delta t)^2-1}{2}\underset{K\to\infty}\to \infty,
\]
whereas, by the Central Limit Theorem, one has the convergence in distribution
\[
\frac{1}{\sqrt{K}}\sum_{k=1}^{K}\Bigl(\frac{|\langle Y,e_k\rangle|^2}{q_k}-1\Bigr)\underset{K\to\infty}\implies \mathcal{N}(0,\frac12).
\]
Consider the following family of bounded continuous functions:
\[
\varphi_{\epsilon,K}=2\min\Bigl(1,\exp\bigl(-\frac{\epsilon}{\sqrt{K}}\sum_{k=1}^{K}\bigl(\frac{|\langle \cdot,e_k\rangle|^2}{q_k}-1\bigr)\bigr)\Bigr)-1,
\]
with arbitrary $\epsilon>0$ and $K\in\N$. Then letting first $K\to\infty$, then $\epsilon\to 0$, one has
\[
\underset{\varphi\in \mathcal{C}_b^0(H,\R),\|\varphi\|_\infty\le 1}\sup~\big|\int \varphi d\nu^{\Delta t,\theta}-\int\varphi d\nu\big|\ge \underset{\epsilon\to 0}\lim~\underset{K\to\infty}\lim~\big|\int \varphi_{\epsilon,K} d\nu^{\Delta t,\theta}-\int\varphi_{\epsilon,K} d\nu\big|=2,
\]
where $\mathcal{C}_b^0(H,\R)$ denotes the space of bounded continuous functions from $H$ to $\R$.
\end{proof}

\subsection{Error analysis of the bias for the invariant distribution}

The objective of this section is to show that the method \eqref{eq:EulerExplicite} has order one of convergence for sampling the invariant distribution.
We rely on the following additional assumption, that is naturally satisfied if $F$ is of $(\beta,2)$-regularity.
\begin{ass}\label{ass:F-Euler}
The mapping $G:H\to H$ is twice Fr\'echet differentiable, with bounded first and second order derivatives.
\end{ass}

Assumption \ref{ass:F-Euler} is a strong assumption that simplifies the convergence analysis, but does not include for instance the case of the Nemytskii operator \eqref{equation:Nemytskii} for $f(x)=-x+\cos(x)$. 
In Section \ref{section:high_order_inv_meas}, we shall consider instead alternate assumptions for higher order estimates for the invariant distribution.

\begin{theo}\label{theo:order1}
Let Assumptions~\ref{ass:AF} and~\ref{ass:F-Euler} be satisfied, and consider the integrator given by~\eqref{eq:EulerExplicite}. Assume that $\varphi:H\to\R$ is twice Fr\'echet differentiable, such that
\[\underset{y\in H}\sup~|D\varphi(y)|+\underset{y\in H}\sup~\|D^2\varphi(y)\|_{\mathcal{L}(H)}\le 1.\]
Then, there exists $C>0$ such that for all $\Delta t\in(0,1]$, one has
\begin{equation}\label{eq:theo-order1}
\big|\int \varphi d\mu^{\Delta t}-\int\varphi d\mu_\star\big|\le C\Delta t.
\end{equation}
\end{theo}

Theorem~\ref{theo:order1} is a generalization to the non-Gaussian case of the error estimate~\eqref{eq:error_Gaussian_C2}. The need to consider sufficiently regular test functions $\varphi$ has been justified by the non-convergence result~\eqref{eq:error_Gaussian_TV}.
The main tool of the proof is the auxiliary function $v$ defined by~\eqref{eq:v}, which solves the Kolmogorov equation~\eqref{eq:Kolmogorov}. The regularity estimates stated in Proposition~\ref{propo:Kolmogorov} play a crucial role.

\begin{proof}[Proof of Theorem~\ref{theo:order1}]
Let $\gamma'\in(0,\gamma)$ be fixed.
To prove that~\eqref{eq:theo-order1} holds, it suffices to establish the following weak error estimate: there exists $C>0$ such that for all $\Delta t\in(0,\Delta t_0]$ and all $N\in\N$, one has
\[
\big|\E[\varphi(Y_N)]-\E[\varphi(Y(t_N))]\big|\le C\Delta t,
\]
with $t_N=N\Delta t$. Passing to the limit $N\to\infty$ then yields~\eqref{eq:theo-order1}.

The strategy employed to study the weak error of numerical schemes, both in finite and infinite dimension is as follows. First the weak error is rewritten in terms of the solution $v$ of the Kolmogorov equation. Second a telescoping sum argument yields the following error expansion:
\begin{align*}
E[\varphi(Y_N)]-\E[\varphi(Y(t_N)]&=\E[v(0,Y_N)]-\E[v(t_N,Y_0)]\\
&=\sum_{n=0}^{N-1}\E\bigl[v(t_N-t_{n+1},Y_{n+1})-v(t_N-t_n,Y_n)\bigr].
\end{align*}
In order to apply the It\^o formula to estimate each term of the sum above, it is convenient to introduce the following continuous-time auxiliary process: for all $n\in\N_0$ and all $t\in[t_n,t_{n+1}]$, set
\[
\tilde{Y}(t)=Y_n+(t-t_n)G(Y_n)+\bigl(W^Q(t)-W^Q(t_n)\bigr).
\]
Observe that $\tilde{Y}(t_n)=Y_n$ for all $n\in\N_0$, and that for $t\in[t_n,t_{n+1}]$, one has
\[
d\tilde{Y}(t)=G(Y_n)dt+dW^Q(t).
\]
Owing to the moment bounds~\eqref{eq:moment-num-Euler} and using the Lipschitz continuity of $F$, the following moment bounds for the auxiliary process $\tilde{Y}$ is obtained: if $\Delta t\le \Delta t_0$, one has
\begin{equation}\label{eq:moment_Ytilde}
\underset{t\ge 0}\sup~\E|\tilde{Y}(t)|\le C(1+\underset{n\ge 0}\sup~\E|Y_n|)\le C(1+\E|Y_0|).
\end{equation}
For each $n\in\N_0$, introduce the auxiliary infinitesimal generator $\tilde{\mathcal{L}}_n^{\Delta t}$, such that
\[
\tilde{\mathcal{L}}_n^{\Delta t}\varphi(y)=\langle G(Y_n),D\varphi(y)\rangle+\frac{1}{2}{\rm Tr}\bigl(QD^2\varphi(y)\bigr).
\]
Applying the It\^o formula and the fact that $v$ solves the Kolmogorov equation~\eqref{eq:Kolmogorov}, for all $n\in\{0,\ldots,N-1\}$, one has
\begin{align*}
\E\bigl[v(t_N-t_{n+1},Y_{n+1})-v(t_N-t_n,Y_n)\bigr]&=\int_{t_n}^{t_{n+1}}\E\bigl[\bigl(\tilde{\mathcal{L}}_n^{\Delta t}-\partial_t\bigr)v(t_N-t,\tilde{Y}(t))\bigr]dt\\
&=\int_{t_n}^{t_{n+1}}\E\bigl[\bigl(\tilde{\mathcal{L}}_n^{\Delta t}-\mathcal{L}\bigr)v(t_N-t,\tilde{Y}(t))\bigr]dt\\
&=\int_{t_n}^{t_{n+1}}\E\bigl[\langle \frac{G(Y_n)-G(\tilde{Y}(t))}{1+\theta\Delta t},Dv(t_N-t,\tilde{Y}(t))\rangle\bigr]dt.
\end{align*}
We introduce the following decomposition of the error:
\begin{align*}
\epsilon_n^{1}&=\int_{t_n}^{t_{n+1}}\E\bigl[\langle G(Y_n)-G(\tilde{Y}(t)),Dv(t_N-t_n,Y_n)\rangle\bigr]dt,\\
\epsilon_n^{2}&=\int_{t_n}^{t_{n+1}}\E\bigl[\langle G(Y_n)-G(\tilde{Y}(t)),Dv(t_N-t_n,\tilde{Y}(t))-Dv(t_N-t_n,Y_n)\rangle\bigr]dt,\\
\epsilon_n^{3}&=\int_{t_n}^{t_{n+1}}\E\bigl[\langle G(Y_n)-G(\tilde{Y}(t)),Dv(t_N-t,\tilde{Y}(t))-Dv(t_N-t_n,\tilde{Y}(t))\rangle\bigr]dt.
\end{align*}

To deal with the first error term $\epsilon_n^{1}$, introduce the auxiliary function $\Psi_n=\langle G(\cdot),Dv(t_N-t_n,Y_n)\rangle$. Using a conditioning argument and applying the It\^o formula, one obtains
\begin{align*}
\E\bigl[\langle G(Y_n)-G(\tilde{Y}(t)),Dv(t_N-t_n,Y_n)\rangle~\big|~Y_n\bigr]&=\E\bigl[\Psi_n(Y_n)-\Psi_n(\tilde{Y}_n(t))~\big|~Y_n\bigr]\\
&=\int_{t_n}^{t}\E\bigl[\tilde{\mathcal{L}}_n^{\Delta t}\Psi_n(\tilde{Y}(s))~\big|~Y_n\bigr]ds.
\end{align*}
Owing to Assumption~\ref{ass:F-Euler} and to Proposition~\ref{propo:Kolmogorov}, it is then straightforward to check that the first and second order derivatives of $\Psi_n$ satisfy
\[
|D\Psi_n(y).h|\le Ce^{-\gamma(t_N-t_n)}|h|\quad,~|D^2\Psi_n(y).(h,k)|\le Ce^{-\gamma(t_N-t_n)}|h||k|.
\]
As a consequence, using the moment bounds~\eqref{eq:moment_Ytilde},
\[
|\epsilon_n^{1}|\le C(1+\E|Y_0|)\Delta t^2e^{-\gamma(t_N-t_n)}.
\]

To deal with the second error term $\epsilon_{n}^{2}$, it suffices to use Proposition~\ref{propo:Kolmogorov} and the Lipschitz continuity of $F$, to obtain
\[
|\epsilon_n^{2}|\le Ce^{-\gamma'(t_N-t_n)}\int_{t_n}^{t}\E|\tilde{Y}(t)-Y_n|^2dt\le C(1+\E|Y_0|^2)\Delta t^2e^{-\gamma'(t_N-t_n)}.
\]
To get the last inequality above, observe that for all $t\in[t_n,t_{n+1}]$,
\begin{equation}\label{eq:accroissement}
\E|\tilde{Y}(t)-Y_n|^2=(t-t_n)^2\E|G(Y_n)|^2+(t-t_n){\rm Tr}(Q) \le C(1+|y|^2)\Delta t,
\end{equation}
using the moment bounds~\eqref{eq:moment-num-Euler}.

To deal with the third error term $\epsilon_n^{3}$, it suffices to use the estimates~\eqref{eq:propo-Kolmogorov2} (see Proposition~\ref{propo:Kolmogorov}) and~\eqref{eq:accroissement}: indeed, one obtains
\[
|\epsilon_n^{3}|\le C\int_{t_n}^{t_{n+1}}e^{-\gamma'(t_N-t)}(t-t_n)(1+\E|Y_0|^2)ds,
\]
using the moment bounds~\eqref{eq:moment-num-Euler}.

Gathering the estimates for $\epsilon_n^1$, $\epsilon_n^2$ and $\epsilon_n^3$ then yields
\begin{align*}
\Big|\E[\varphi(Y_N)]-\E[\varphi(Y(t_N)]\Big|&\le \sum_{n=0}^{N-1}\Big|\E\bigl[v(t_N-t_{n+1},Y_{n+1})-v(t_N-t_n,Y_n)\bigr]\Big|\\
&\le \sum_{n=0}^{N-1}\bigl(|\epsilon_n^1|+|\epsilon_n^2|+|\epsilon_n^3|\bigr)\\
&\le C(1+\E|Y_0|^2)\Delta t\int_{0}^{t_N}e^{-\gamma'(t_N-t)}dt\\
&\le C(1+\E|Y_0|^2)\Delta t\int_{0}^{\infty}e^{-\gamma' t}dt.
\end{align*}
This concludes the proof of Theorem~\ref{theo:order1}.
\end{proof}

\section{High order convergence analysis}
\label{section:high_order_inv_meas}

The aim of this section is to create, under stronger regularity assumptions, high order methods for sampling $\mu_\star$, given by \eqref{eq:def_mu_star}.
We consider one step integrators with postprocessors,
\[
Y_{n+1}^{\Delta t}=\psi^{\Delta t}(Y_n^{\Delta t}),
\quad
\overline{Y}_n^{\Delta t}=\overline{\psi}^{\Delta t}(Y_n^{\Delta t}).
\]
We consider a vector field $F\in \mathfrak{X}^{4,\beta}_P$ and test functions $\varphi\in\mathcal{C}^{6,\beta}_P$.
Recall that the dynamics $\bigl(Y(t)\bigr)_{t\ge 0}$ to be discretized is given by the preconditioned SPDE~\eqref{eq:precondSPDE}. This process is ergodic and mixing, and its unique invariant distribution is $\mu_\star$.
Thus, we assume a similar property for the integrator.
We consider the following assumptions.
\begin{ass}
\label{ass:ergodic_num}
For some $\Delta t_0>0$, if $\Delta t\in(0,\Delta t_0]$, the Markov process $\bigl(Y_n^{\Delta t}\bigr)_{n\ge 0}$ is ergodic and mixing. Its unique invariant distribution, denoted by $\mu^{\Delta t}$, satisfies
\[
\E[\varphi(Y_n)]\underset{n\to\infty}\to \int \varphi d\mu^{\Delta t},
\]
for all bounded and continuous test functions $\varphi:H\to \R$, and any arbitrary initial condition $y\in H$.
\end{ass}

\begin{ass}
\label{ass:bounded_mom}
The integrator has bounded moments of all order: for any $\kappa\in\N$, there exists $C>0$ such that
\[
\underset{\Delta t\in(0,\Delta t_0]}\sup~\underset{n\ge 0}\sup~\E[|Y_n^{\Delta t}|_\beta^\kappa]\le C(1+|Y_0|_\beta^\kappa).
\]
\end{ass}

Note that Assumption \ref{ass:bounded_mom} provides moment bounds for the numerical invariant distribution $\mu^{\Delta t}$:
\[
\underset{\Delta t\in(0,\Delta t_0]}\sup~\int |y|_\beta^\kappa d\mu^{\Delta t}(y)\le C.
\]

In Section \ref{sec:analysis_high_order_sampling}, we perform an analysis for a general scheme relying on Assumptions \ref{ass:ergodic_num} and \ref{ass:bounded_mom}, while in Section \ref{sec:LM}, we focus in particular on the following method \eqref{eq:EulerExplicitePostprocessed} of second order for the invariant distribution.
It corresponds to the Leimkuhler-Matthews method \cite{Leimkuhler13rco, Leimkuhler14otl} applied to the preprocessed dynamics \eqref{eq:precondSPDE} and reformulated as a postprocessed integrator \cite{Vilmart15pif},
\begin{equation}\label{eq:EulerExplicitePostprocessed}
\begin{aligned}
Y_{n+1}&=Y_n+\Delta tG(Y_n+\frac12\Delta W_n^Q)+\Delta W_n^Q,\\
\overline{Y}_n&=Y_n+\frac12\Delta W_n^Q,
\end{aligned}
\end{equation}
and on its spectral discretisation
\begin{align*}
Y_{n+1}^K&=Y_n^K+\Delta tG^K(Y_n^K+\frac12\Delta W_n^{Q,K})+\Delta W_n^{Q,K},\\
\overline{Y}_n^K&=Y_n^K+\frac12\Delta W_n^{Q,K}.
\end{align*}
Note that, under Assumption \ref{ass:AF}, the scheme~\eqref{eq:EulerExplicitePostprocessed} naturally satisfies Assumptions~\ref{ass:ergodic_num} and~\ref{ass:bounded_mom}, which is shown by using the same techniques as in the proof of the moment bounds~\eqref{eq:moment-num-Euler} from Lemma~\ref{lem:ergoYn} (see also \cite[Lemma 2.2]{Milstein04snf}).

\subsection{Convergence analysis for high order methods}
\label{sec:analysis_high_order_sampling}

The objective of this section is to exhibit order conditions to have methods of order $2$.

First, Lemma~\ref{lem:order1} is an alternative approach compared to Section \ref{section:simple_CV_analysis} to check that a method is of order $1$.
\begin{lemma}\label{lem:order1}
Under Assumptions \ref{ass:ergodic_num} and \ref{ass:bounded_mom}, assume that, for any $\varphi\in\mathcal{C}^{4,\beta}_P$, one has the Taylor-type expansion
\[
\E\bigl[\varphi\bigl(\psi^{\Delta t}(y)\bigr)\bigr]-\varphi(y)=\Delta t\mathcal{L}\varphi(y)+\Delta t^2R_1(\varphi,y,\Delta t),
\]
for all $\Delta t$ assumed small enough,
with $\underset{\Delta t\in(0,\Delta t_0]}\sup~|R_1(\varphi,y,\Delta t)|\le C(1+|y|_\beta^\kappa)$ for some $p\in\N$ and $C>0$.
Then, for $\Delta t \rightarrow 0$ and $\varphi\in\mathcal{C}^{4,\beta}_P$,
\begin{equation}\label{eq:expansion_order1}
\int\varphi d\mu^{\Delta t}-\int\varphi d\mu_\star={\rm O}(\Delta t).
\end{equation}
\end{lemma}

\begin{proof}[Proof of Lemma~\ref{lem:order1}]
Since $\mu^{\Delta t}$ is the invariant distribution of the Markov chain $\bigl(Y_n^{\Delta t}\bigr)_{n\ge 0}$, one has the following equality: for all $\Delta t\in(0,\Delta t_0]$,
\begin{equation}\label{eq:equilibrium}
\int \Bigl(\E\bigl[\varphi\bigl(\psi^{\Delta t}(y)\bigr)\bigr]-\varphi(y)\Bigr)d\mu^{\Delta t}(y)=0.
\end{equation}
Then, one obtains
\[
0=\int \mathcal{L}\varphi(y)d\mu^{\Delta t}(y)+\Delta t\int R_1(\varphi,y,\Delta t) d\mu^{\Delta t}(y).
\]
Choosing $\varphi=\Psi$ then yields
\[
\int\varphi d\mu^{\Delta t}-\int\varphi d\mu_\star=\Delta t\int R_1(\varphi,y,\Delta t)d\mu^{\Delta t}(y)={\rm O}(\Delta t),
\]
owing to the moment bounds for the numerical invariant distribution $\mu^{\Delta t}$.
\end{proof}

Using the same approach then yields the following result.
\begin{propo}\label{propo:order2}
Under Assumptions \ref{ass:ergodic_num} and \ref{ass:bounded_mom}, assume that, for any $\varphi\in\mathcal{C}^{6,\beta}_P$, one has the expansion
\[
\E\bigl[\varphi\bigl(\psi^{\Delta t}(y)\bigr)\bigr]-\varphi(y)=\Delta t\mathcal{L}\varphi(y) +\Delta t^2\mathcal{A}_1\varphi(y)+\Delta t^3R_2(\varphi,y,\Delta t),
\]
for all $\Delta t$ assumed small enough,
where $\mathcal{A}_1\varphi\in\mathcal{C}^{2,\beta}_P$ and the following estimate is satisfied for some $p\in\N$ and $C>0$,
\[\underset{\Delta t\in(0,\Delta t_0]}\sup~\bigl(|\mathcal{A}_1\varphi(y)|+|R_2(\varphi,y,\Delta t)|\bigr)\le C(1+|y|_\beta^\kappa),\]
where $p\in\N$ and $C>0$.
Then, if $\varphi\in\mathcal{C}^{6,\beta}_P$, one has for $\Delta t \rightarrow 0$,
\begin{equation}\label{eq:expansion_order2}
\int\varphi d\mu^{\Delta t}-\int\varphi d\mu_\star=\Delta t\int \mathcal{A}_1\Psi(y)d\mu_\star(y)+{\rm O}(\Delta t^2),
\end{equation}
where $\Psi$ is given by the Poisson equation \eqref{eq:Poisson}.
Moreover, assume further that for any $\varphi\in\mathcal{C}^{6,\beta}_P$, one has the error expansion
\[
\E\bigl[\varphi\bigl(\overline{\psi}^{\Delta t}(y)\bigr)\bigr]-\varphi(y)=\Delta t\overline{\mathcal{A}}_1\varphi(y)+\Delta t^2\overline{R}_2(\varphi,y,\Delta t),
\]
with $\overline{\mathcal{A}}_1\varphi\in\mathcal{C}^{4,\beta}_P$, $\overline{\mathcal{A}}_11=0$, and the estimate for some $p\in\N$ and $C>0$,
\[\underset{\Delta t\in(0,\Delta t_0]}\sup~\bigl(|\overline{\mathcal{A}}_1\varphi(y)|+|\overline{R}_2(\varphi,y,\Delta t)|\bigr)\le C(1+|y|_\beta^\kappa).\]
Then, if $\varphi\in\mathcal{C}^{6,\beta}_P$, one has
\begin{equation}\label{eq:expansion_order2-postproc}
\int\varphi d\overline{\mu}^{\Delta t}-\int\varphi d\mu_\star=\Delta t\int \bigl(\mathcal{A}_1+[\mathcal{L},\overline{\mathcal{A}}_1]\bigr)\Psi(y)d\mu_\star(y)+{\rm O}(\Delta t^2),
\end{equation}
where $[\mathcal{L},\overline{\mathcal{A}}_1]=\mathcal{L}\overline{\mathcal{A}}_1-\overline{\mathcal{A}}_1\mathcal{L}$ denotes the commutator of the operators $\mathcal{L}$ and $\overline{\mathcal{A}}_1$. 
\end{propo}

\begin{proof}[Proof of Proposition~\ref{propo:order2}]
The proof of~\eqref{eq:expansion_order2} follows the same steps as the proof of~\eqref{eq:expansion_order1}: using the identity~\eqref{eq:equilibrium} with $\varphi=\Psi$, one obtains
\begin{align*}
\int\varphi d\mu^{\Delta t}-\int\varphi d\mu_\star&=\Delta t\int\mathcal{A}_1\Psi d\mu^{\Delta t}+{\rm O}(\Delta t^2)\\
&=\Delta t\int\mathcal{A}_1\Psi d\mu_\star+{\rm O}(\Delta t^2)
\end{align*}
using~\eqref{eq:expansion_order1} for the function $\mathcal{A}_1\Psi\in\mathcal{C}^{4,\beta}_P$ in the last step.
To prove~\eqref{eq:expansion_order2-postproc}, observe that
\begin{align*}
\int \varphi d\overline{\mu}^{\Delta t}-\int\varphi d\mu^{\Delta t}&=\int \Bigl(\E\bigl[\varphi\bigl(\overline{\psi}^{\Delta t}(y)\bigr)\bigr]-\varphi(y)\Bigr)d\mu^{\Delta t}(y)\\
&=\Delta t\int \overline{\mathcal{A}}_1\varphi d\mu^{\Delta t}+{\rm O}(\Delta t^2)\\
&=\Delta t\int \overline{\mathcal{A}}_1\varphi d\mu_\star+{\rm O}(\Delta t^2),
\end{align*}
using~\eqref{eq:expansion_order1} for the function $\overline{\mathcal{A}}_1\varphi\in\mathcal{C}^{4,\beta}_P$ in the last step.
Combined with~\eqref{eq:expansion_order2}, this yields
\begin{align*}
\int\varphi d\overline{\mu}^{\Delta t}-\int\varphi d\mu_\star&=\Delta t\int \bigl(\mathcal{A}_1\Psi-\overline{\mathcal{A}}_1\varphi\bigr)d\mu_\star+{\rm O}(\Delta t^2)\\
&=\Delta t\int \bigl(\mathcal{A}_1-\overline{\mathcal{A}}_1\mathcal{L}\bigr)\Psi d\mu_\star+{\rm O}(\Delta t^2)\\
&=\Delta t\int \bigl(\mathcal{A}_1+[\mathcal{L},\overline{\mathcal{A}}_1]\bigr)\Psi(y)d\mu_\star(y)+{\rm O}(\Delta t^2),
\end{align*}
using $\varphi=-\mathcal{L}\Psi+\int\varphi d\mu_\star$, the assumption that $\overline{\mathcal{A}}_1 1=0$, and, in the last step, the property $\int \mathcal{L}\varphi d\mu_\star=0$.
\end{proof}

As an immediate consequence of Proposition~\ref{propo:order2}, one may state algebraic conditions for order $2$ as follows.
\begin{cor}\label{cor:order2}
Under the Assumptions of Proposition~\ref{propo:order2} and if $\int\mathcal{A}_1\varphi d\mu_\star=0$ for all $\varphi\in\mathcal{C}^{6,\beta}_P$, then, for all $\varphi\in\mathcal{C}^{6,\beta}_P$, one has
\[
\int\varphi d\mu^{\Delta t}-\int\varphi d\mu_\star={\rm O}(\Delta t^2).
\]
Alternatively, assume that $\int\bigl(\mathcal{A}_1+[\mathcal{L},\overline{\mathcal{A}}_1]\bigr)\varphi d\mu_\star=0$ for all $\varphi\in\mathcal{C}^{6,\beta}_P$. Then, for all $\varphi\in\mathcal{C}^{6,\beta}_P$, one has
\[
\int\varphi d\overline{\mu}^{\Delta t}-\int\varphi d\mu_\star={\rm O}(\Delta t^2).
\]
\end{cor}

\subsection{Construction of second order methods}
\label{sec:LM}

The main result of this section is the following convergence theorem of order two for the new method \eqref{eq:EulerExplicitePostprocessed}.

\begin{theo}
\label{theorem:high_order_inv_measure}
Let the integrator with postprocessing be given by~\eqref{eq:EulerExplicitePostprocessed}.
Under Assumption \ref{ass:AF} and assuming $F\in \mathfrak{X}^{4,\beta}_P$, for all $\varphi\in \mathcal{C}^{6,\beta}_P$ and all $\Delta t$ small enough, the following second order estimate for sampling the invariant measure of \eqref{eq:semilin} holds,
\[
\int\varphi d\overline{\mu}^{\Delta t}-\int\varphi d\mu_\star={\rm O}(\Delta t^2).
\]
Moreover, in the Gaussian case $F=0$, the method is exact: $\overline{\mu}^{\Delta t}=\nu$.
\end{theo}

The proof of Theorem \ref{theorem:high_order_inv_measure} relies on Corollary~\ref{cor:order2} and the lemmas below on the weak Taylor expansion of method and its postprocessor \eqref{eq:EulerExplicitePostprocessed}, in the spirit of the analysis of \cite[Sect.\ts 8]{Abdulle14hon} and \cite[Sect.\ts 7]{Vilmart15pif} in the finite dimensional context. A key ingredient is to perform the calculations for the finite-dimensional dynamics in $H^K$, and then take the limit $K\rightarrow\infty$.

\begin{remark}
The analysis could be extended for the creation of methods of any high order for the invariant distribution in the spirit of \cite{Abdulle14hon} in the finite dimensional context. As the calculations become tedious, we mention that the algebraic formalism of exotic Butcher series \cite{Bronasco22ebs, Bronasco22cef, Laurent21ata, Laurent20eab} could be extended to the present context and used for simplifying the calculations of the Taylor expansions.
\end{remark}

\begin{lemma}
\label{lemma:second_order_exp}
Under the assumptions of Theorem \ref{theorem:high_order_inv_measure} and for all $\varphi\in \mathcal{C}^{6,\beta}_P$, the local weak expansion of the integrator \eqref{eq:EulerExplicitePostprocessed} is
$$\IE(\varphi(Y_{n+1}^K)|Y_{n}^K=x)= \varphi(x) + \Delta t{\mathcal L} \varphi(x) + \Delta t^2 {\mathcal A}_{1}\varphi(x)+\Delta t^3 R_{2}(\varphi,x,\Delta t)$$ 
with
$$
\begin{aligned}
{\mathcal A}_{1} \varphi&=\frac 12D^2\varphi.(G^K,G^K) + \frac 12 \sum_{i=1} q_iD^3\varphi.(e_i, e_i,G^K) + \frac{1}8 \sum_{i,j=1} q_i q_jD^4\varphi.(e_i,e_i,e_j,e_j)\\
&+\frac18 D\varphi. \sum_{i=1} q_iD^2G^K.(e_i,e_i) + \frac 12 \sum_{i=1} q_iD^2\varphi.(DG^K.e_i, e_i),
\end{aligned}
$$
where $\bigl(e_k\bigr)_{k\in\N}$ is the complete orthonormal system of $H$ considered in Definition~\ref{defi:QAnu} and the remainder satisfies the bound
$$\E|R_{2}(\varphi,x,\Delta t)|\leq C(1+|x|_\beta^\kappa).$$
\end{lemma}

\begin{proof}
Following the lines of \cite[Sect.\ts 7]{Vilmart15pif} (where the $q_i$ coefficients were not considered), we perform a Taylor expansion of $\varphi(Y_{n+1})$ around $Y_n$ up to order 6 and a Taylor expansion of $G(Y_{n+1})$ up to order 4, which yields the desired expansion with the remainder:
$$
\begin{aligned}
R_{2}(\varphi,x,\Delta t) &=
\IE
\int_0^1 \frac{(1-\theta)^3}{3!} 
D\varphi(x).D^4G^K(x+\frac \theta2\Delta W_{N}^Q).
(\frac{\frac12\Delta W_{N}^Q}{\sqrt{\Delta t}})^4 d\theta\\
&+
\IE
\int_0^1 \frac{(1-\theta)^5}{5!} 
D^6\varphi(x+\theta\Delta Y_n^K).
(\frac{\Delta Y_n^K}{\sqrt{\Delta t}})^6 d\theta,
\end{aligned}$$
where
$$\Delta Y_n^K=Y_{n+1}^K-Y_{n}^K=\Delta tG^K(Y_{n}^K+\frac12\Delta W_{n}^{Q})+\Delta W_{n}^{Q}.
$$
We first observe using \eqref{eq:alphasup} and $G=-x+QF(x)$:
\[
\E|\Delta W_n^Q|_\beta^\kappa=\E|(-A)^{\beta/2}Q^{1/2}\Delta W_n|^\kappa\leq C\Delta t^{\kappa/2},\quad
|G(x)|_\beta\leq C(1+|x|_\beta^\kappa).
\]
Thus, the increment satisfies
\[
\E|\frac{\Delta Y_n^K}{\sqrt{\Delta t}}|\leq C(1+|x|_\beta^\kappa).
\]
The bound on the remainder is a consequence of the regularity assumptions $\varphi\in \mathcal{C}^{6,\beta}_P$ and $F\in \mathfrak{X}^{4,\beta}_P$.
\end{proof}

\begin{lemma}
\label{lemma:second_order_exp_pp}
Under the assumptions of Theorem \ref{theorem:high_order_inv_measure} and for all $\varphi\in \mathcal{C}^{6,\beta}_P$, the local weak expansion of the postprocessor integrator \eqref{eq:EulerExplicitePostprocessed} is
$$
\IE(\varphi(\overline{Y}_{n}^K)|Y_{n}^K=x)= \varphi(x) + \Delta t\overline{\mathcal{A}}_1 \varphi(x) + \Delta t^2 \overline{R}_{1}(\varphi,x,\Delta t),$$
where
$$
\overline{\mathcal{A}}_{1}\varphi = \frac18 \sum_{i=1} q_iD^2\varphi.(e_i,e_i),
$$
and the remainder satisfies the bound
$$\E[|\overline{R}_{1}(\varphi,x,\Delta t)|]\leq C (1+|x|_\beta^\kappa).$$
\end{lemma}

\begin{proof}
The remainder is given by
$$
\begin{aligned}
\overline{R}_{1}(\varphi,x,\Delta t) &=
\IE
\int_0^1 \frac{(1-\theta)^3}{3!} 
D^4\varphi(x+\theta\Delta W_n^Q).
(\frac{\Delta W_n^Q}{\sqrt{\Delta t}})^4 d\theta.
\end{aligned}$$
and the proof concludes analogously to that of Lemma \ref{lemma:second_order_exp}.
\end{proof}

\begin{lemma}
The operators $\mathcal{A}_1$ and $\overline{\mathcal{A}}_{1}$ from Lemma \ref{lemma:second_order_exp} and \ref{lemma:second_order_exp_pp} satisfy the second order condition of Proposition \ref{propo:order2}:
$$
\left\langle\mathcal{A}_{1}\varphi + [\mathcal L, \overline{\mathcal{A}}_{1}] \varphi \right\rangle =0.
$$
\end{lemma}

\begin{proof}
Considering the notations $\left\langle \varphi \right\rangle_{\mu_{\star}} =\int_{H} \varphi(x) d\mu_{\star}(x)$, $G_j^K=\left\langle G^K, e_j \right\rangle$, $\partial_j\varphi = \varphi'e_j$, and using repeatedly the integration by parts formula, 
\[
\left\langle \partial_j\varphi_1 \, \varphi_2 \right\rangle_{\mu_{\star}}
=
\left\langle -\varphi_1 \partial_j\varphi_2\, -2\varphi_1\varphi_2 q_j^{-1} G_j^K \right\rangle_{\mu_{\star}}
\]
a calculation yields
$$
\begin{aligned}
\left\langle \sum_{i,j=1} q_i q_j D^4 \varphi.(e_i,e_i,e_j,e_j) \right\rangle_{\mu_{\star}} &= \left\langle -2\sum_{i=1} q_i D^3\varphi.(e_i,e_i,G^K) \right\rangle_{\mu_{\star}},\\
\left\langle \sum_{i=1} q_i D^3 \varphi.(e_i, e_i,G^K) \right\rangle_{\mu_{\star}} &= 
\left\langle -\sum_{i=1} q_i D^2 \varphi.(DG^K.e_i,e_i)-2 D^2 \varphi.(G^K,G^K) \right\rangle_{\mu_{\star}}.
\end{aligned}
$$
We deduce
$$
\begin{aligned}
\left\langle\mathcal A_{1}\varphi \right\rangle&=
\left\langle
\frac18 \sum_{i=1} q_i D\varphi.(D^2G^K.(e_i,e_i)) + \frac14\sum_{i=1} q_iD^2\varphi.(DG^K.e_i,e_i)
\right\rangle
\end{aligned}
$$
Using the calculation for the commutator operator,
$$
[\mathcal L, \overline{\mathcal{A}}_{1}] \varphi = 
-\frac18 \sum_{i=1} q_i D\varphi.(D^2G^K.(e_i,e_i)) - \frac14\sum_{i=1} q_iD^2\varphi.(DG^K.e_i,e_i),
$$
this concludes the proof.
\end{proof}

\begin{proof}[Proof of Theorem \ref{theorem:high_order_inv_measure}]
It remains to check the statement that in the Gaussian case $F=0$, one finds $\overline{\mu}^{\Delta t}=\nu$ for all $\Delta t$. First, if $\bigl(Y_n\bigr)_{n\ge 0}$ is given by~\eqref{eq:EulerExplicitePostprocessed}, then
\[
Y_{n+1}=(1-\Delta t)Y_n+(1-\frac{\Delta t}{2})\Delta W_n^Q.
\]
Thus, if $\Delta t<1$, the numerical invariant distributions are Gaussian: $\mu^{\Delta t}=\mathcal{N}(0,\frac{\sigma(\Delta t)^2}{2}Q)$, with
\[
\sigma(\Delta t)^2=\frac{2\Delta t(1-\frac{\Delta t}{2})^2}{1-(1-\Delta t)^2}=1-\frac{\Delta t}{2},
\]
and $\overline{\mu}^{\Delta t}=\mathcal{N}(0,\frac{\sigma(\Delta t)^2+\frac{\Delta t}{2}}{2}Q)=\mathcal{N}(0,\frac12 Q)=\nu$.
\end{proof}

\begin{remark}
We emphasize that alternatively to method \eqref{eq:EulerExplicitePostprocessed}, one can simply consider weak order 2 integrators applied to the preconditioned dynamics \eqref{eq:precondSPDE}, for instance, the Runge-Kutta 2 (RK2) method with two evaluations of $G$ per timestep,
\begin{equation}\label{eq:RK2}
\begin{aligned}
\hat{Y}_n&=Y_n+\Delta tG(Y_n)+\Delta W_n^Q,\\
Y_{n+1}&=Y_n+\frac{\Delta t}{2}\bigl(G(Y_n)+G(\hat{Y}_n)\bigr)+\Delta W_n^Q.
\end{aligned}
\end{equation}
Note, however, that, in contrast to method \eqref{eq:EulerExplicitePostprocessed}, method \eqref{eq:RK2} is not exact for the invariant distribution in the Gaussian case $F=0$.
Alternatively, one can also consider the postprocessed integrator from \cite{Brehier_Vilmart:16} based on the semilinear implicit Euler method
\begin{equation}\label{eq:EulerImplicitePostprocessed}
\begin{aligned}
Y_{n+1}&=\frac{1}{1+\Delta t}Y_n+\frac{\Delta t}{1+\Delta t}QF\bigl(Y_n+\frac{1}{2(1+\Delta t)}\Delta W_n^Q\bigr)+\frac{1}{1+\Delta t}\Delta W_n^Q,\\
\overline{Y}_n&=Y_n+\frac{1}{2(1+\frac{\Delta t}{2})^{\frac12}}\Delta W_n^Q.
\end{aligned}
\end{equation}
An advantage of method \eqref{eq:EulerImplicitePostprocessed}, with only one evaluation of $F$ per timestep, is that it has no timestep restriction for stability and it is exact for the invariant distribution in the Gaussian case $F=0$.
\end{remark}

\section{Numerical experiments}
\label{section:num}
\begin{figure}[t]
	\begin{center}
		\includegraphics[scale=0.45]{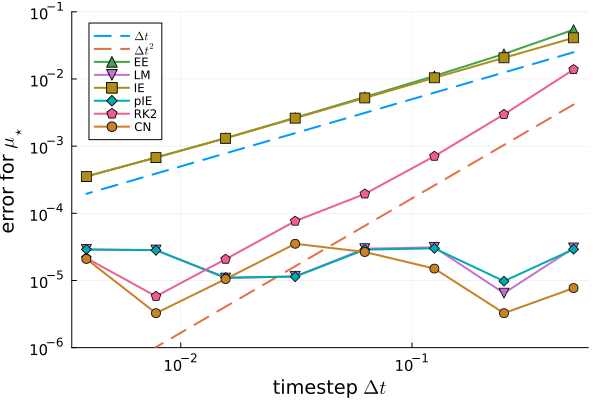}
		\end{center}
	\caption{Error for sampling the invariant distribution of the preconditioned SPDE \eqref{eq:precondSPDE} in the spatial domain $[0,1]$ with $A=\Delta$ and $F(x)=0$ using the explicit and implicit Euler methods EE and IE (\eqref{eq:thetascheme} with $\theta=0$ and $\theta=1$), their postprocessed counterparts the Leimkuhler-Matthews method LM \eqref{eq:EulerExplicitePostprocessed} and pIE \eqref{eq:EulerImplicitePostprocessed}, the RK2 method \eqref{eq:RK2} and the Crank-Nicholson scheme CN \eqref{eq:thetascheme} with $\theta=\frac{1}{2}$, with the test function $\varphi(x)=\exp(-\norme{x}_{L^2}^2)$ and $M=10^8$ trajectories.}
	\label{figure:Error_Inv_meas_no_F}
\end{figure}

Let us now illustrate numerically the theoretical results.
We give error curves for the new methods of the article for the preconditioned SDPE \eqref{eq:precondSPDE} to confirm the theoretical orders of convergence.
Then, we observe the effect of different preconditioning on the order of convergence and introduce an alternative preconditioner based on Chebyshev polynomials.

We first consider the postprocessed explicit and implicit Euler methods \eqref{eq:EulerExplicitePostprocessed}-\eqref{eq:EulerImplicitePostprocessed} for solving the preconditioned SPDE \eqref{eq:precondSPDE} in the spatial domain $[0,1]$ with $A=\Delta$.
We apply a standard finite difference approximation of the Laplacian with the stepsize $\Delta x=0.02$, the final time $T=10$ which is sufficient to reach equilibrium, and different timesteps $\Delta t$ to observe the rates of convergence.
We consider $M=10^8$ trajectories (using 450 CPUs) and plot an estimate of the error for the invariant distribution by using the standard mean estimator:
\[\overline{J}_{\Delta t}=\frac{1}{M}\sum_{m=1}^M \varphi(Y_N^{(m)})\approx \E[\varphi(Y_N)],\]
where $Y_N^{(m)}$ is the $m$-th realisation of the integrator at time $T$ and we consider the test function $\varphi(x)=\exp(-\norme{x}_{L^2}^2)$.
We compare this approximation with a reference value for $\int_H\varphi d\mu_\star$ given by the approximation $\overline{J}_{2^{-9}}$ with the Leimkuhler-Matthews method \eqref{eq:EulerExplicitePostprocessed}.
We emphasize that the approach could be combined with variance reduction techniques, such as Multi-Level Monte-Carlo methods, though this is out of the scope of the present paper.

We observe in Figure \ref{figure:Error_Inv_meas_no_F} the Gaussian case where $F=0$. The integrators EE, IE, and RK2 have the expected order of convergence for the invariant distribution. Moreover, the other integrators have no bias for the invariant distribution and we observe only the Monte-Carlo error.
We repeat the same experiment, this time with the Lipschitz nonlinearity $F(x)=-x+\cos(x)$. We observe the predicted order for the invariant distribution of the methods in Figure \ref{figure:Error_Inv_meas}, which illustrates Theorem \ref{theorem:high_order_inv_measure}. It is remarkable that the Leimkuhler-Matthews scheme with preprocessing \eqref{eq:EulerExplicitePostprocessed} yields the best accuracy among the tested methods.
\begin{figure}[t]
		\begin{center}
		\includegraphics[scale=0.45]{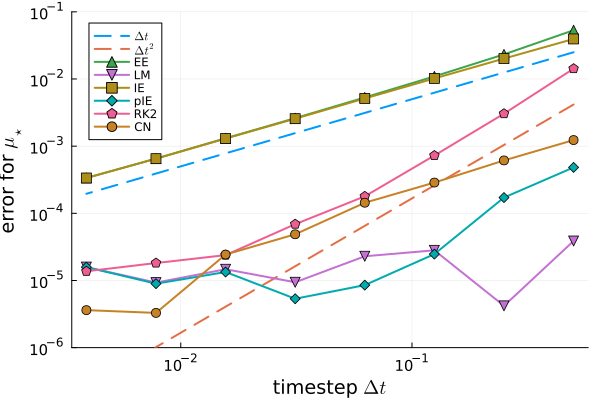}
		\end{center}
	\caption{Error for sampling the invariant distribution of the preconditioned SPDE \eqref{eq:precondSPDE} in the spatial domain $[0,1]$ with $A=\Delta$ and $F(x)=-x+\cos(x)$ using the explicit and implicit Euler methods EE and IE (\eqref{eq:thetascheme} with $\theta=0$ and $\theta=1$), their postprocessed counterparts the Leimkuhler-Matthews method LM \eqref{eq:EulerExplicitePostprocessed} and pIE \eqref{eq:EulerImplicitePostprocessed}, the RK2 method \eqref{eq:RK2} and the Crank-Nicholson scheme CN \eqref{eq:thetascheme} with $\theta=\frac{1}{2}$, with the test function $\varphi(x)=\exp(-\norme{x}_{L^2}^2)$ and $M=10^8$ trajectories.}
	\label{figure:Error_Inv_meas}
\end{figure}

\begin{figure}[t!]
		\begin{center}
		\includegraphics[scale=0.45]{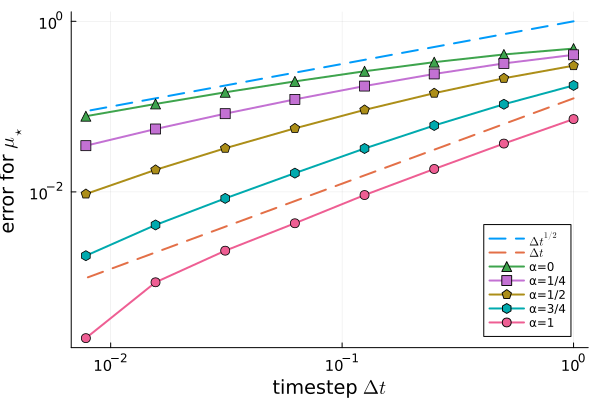}
		\end{center}
	\caption{Error of the integrator \eqref{equation:linear_implicit_Euler_general_preconditioning} for sampling the invariant distribution of the preconditioned SPDE \eqref{eq:precondSPDE-general} in the spatial domain $[0,1]$ with $A=\Delta$, $F(x)=-x+\cos(x)$, $\mathcal{P}=(-A)^{-\alpha}$ for different $\alpha$, the test function $\varphi(x)=\exp(-\norme{x}_{L^2}^2)$ and $M=10^7$ trajectories.}
	\label{figure:Error_preconditioning}
\end{figure}
Let us now observe the behaviour of a given scheme for approximating the preconditioned dynamics \eqref{eq:precondSPDE-general} with $\mathcal{P}=(-A)^{-\alpha}$ for different $\alpha\in [0,1]$ and the nonlinearity $F(x)=-x+\cos(x)$. We apply the following time-discretisation with the same finite difference approximation as previously:
\begin{equation}
\label{equation:linear_implicit_Euler_general_preconditioning}
Y_{n+1}=(I-\Delta t \mathcal{P} A)^{-1}(Y_n+\Delta t \mathcal{P} F(Y_n)+\Delta W_n^{\mathcal{P}}).
\end{equation}
The reference values are given by the integrator \eqref{equation:linear_implicit_Euler_general_preconditioning} with $\Delta_t=2^{-8}$ and $\alpha=1$. We observe on Figure \ref{figure:Error_preconditioning} that the preconditioner improves the order of convergence from $1/2$ to $1$ as $\alpha$ goes from $0$ to $1$, to emphasise the advantage of the preconditioner enhancing the convergence rate, here illustrated with the method  \eqref{equation:linear_implicit_Euler_general_preconditioning}.

\smallskip

\noindent \textbf{Acknowledgements.}\
C.-E.B. acknowledges the support from the program ANR-19-CE40-0016 (SIMALIN) and ANR-24-CE40-3786 (SOS2ID).
A.B.L. is grateful for the support of the French programs ANR-11-LABX-0020-0 (Labex Lebesgue) and ANR-25-CE40-2862-01 (MaStoC).
A.D. acknowledges the support from the program ANR-11-LABX-0020-0 (Labex Lebesgue) and ANR-19-CE40-0019 (ADA).
G.V. was partially supported by the Swiss National Science Foundation, projects No 200020\_214819 and No. 200020\_192129.
The computations were performed at the University of Geneva on the Baobab cluster using the Julia
programming language.

\bibliographystyle{abbrv}
\bibliography{biblio,Ma_Bibliographie,gilles}

\end{document}